\documentclass{article}
\usepackage[utf8]{inputenc}

\usepackage{amsmath,mathtools,amsthm,amssymb, xcolor}
\usepackage{thm-restate}
\usepackage{cite}
\usepackage{tikz}
\usepackage{enumitem}
\usetikzlibrary{calc}
\usepackage{fullpage}
\usepackage{makecell}

\usepackage[pdftex,hypertexnames=false,linktocpage=true]{hyperref}
\hypersetup{colorlinks=true,linkcolor=black,anchorcolor=blue,citecolor=black,filecolor=blue,urlcolor=black,bookmarksnumbered=true,pdfview=FitB}

\newtheorem{theorem}{Theorem}[section]
\newtheorem{corollary}[theorem]{Corollary}
\newtheorem{lemma}[theorem]{Lemma}

\newtheorem{problem}[theorem]{Problem}
\newtheorem{conjecture}[theorem]{Conjecture}

\theoremstyle{definition}
\newtheorem{definition}{Definition}[section]

\newcommand{\oururl}{\url{http://lidicky.name/pub/rt/}}

\newcommand{\bu}{\mathbf{u}}
\newcommand{\bx}{\mathbf{x}}
\newcommand{\calF}{\mathcal{F}}

\tikzset{
vtx/.style={inner sep=1.7pt, outer sep=0pt, circle, fill=black,draw=black}, 
}

\tikzset{
xvtx/.style={inner sep=1.7pt, outer sep=0pt, circle,draw=black}, 
}

\title{Weighted Tur\'an theorems with applications to Ramsey-Tur\'an type of problems}

\author{
J\'ozsef Balogh\thanks{
Department of Mathematical Sciences,
University of Illinois,
Urbana, IL, USA.
E-mail: {\tt jobal@illinois.edu}.
Balogh was supported in part by NSF grants DMS-1764123, RTG DMS-1937241 and FRG DMS-2152488, the Arnold O. Beckman Research Award (UIUC Campus Research Board RB 22000), and the Langan Scholar Fund (UIUC).}
\and 
Domagoj Brada\v{c}\thanks{Department of Mathematics, ETH, Z\"urich, Switzerland. Email:
\href{mailto:domagoj.bradac@math.ethz.ch} {\nolinkurl{domagoj.bradac@math.ethz.ch}}.
Research is supported in part by SNSF grant 200021\_196965.}{}
\and
Bernard Lidick\'y\thanks{
Department of Mathematics, Iowa State University, Ames, IA, USA. E-mail: {\tt lidicky@iastate.edu}. Research of this author is supported in part by NSF FRG grant DMS-2152490 and Scott Hanna fellowship.
}
}

\date{}

\begin{document}

\maketitle

\begin{abstract}
We study extensions  of  Tur\'an Theorem  in edge-weighted settings. A particular case of interest is when constraints on the weight of an  edge come from the order of the largest clique containing it.
These problems are motivated by Ramsey-Tur\'an type problems. Some of our proofs are based on the method of graph Lagrangians, while the other proofs use flag algebras. Using these results, we prove several new upper bounds on the  Ramsey-Tur\'an density of cliques. Other applications of our results  are  in a recent paper of Balogh, Chen, McCourt and Murley. 
\end{abstract}
\emph{
Keywords: Ramsey-Tur\'an, weighted Tur\'an, weighted graphs, flag algebras.
}

\section{Introduction}
\subsection{Ramsey-Tur\'{a}n theory}\label{sec-intro-ramsey-turan}
The maximum number of edges in a $K_{k}$-free graph on $n$ vertices is achieved by the balanced complete $(k-1)$-partite graph, which is  known as the {\it Tur\'an Graph} $T(n,k-1)$. Observe that $T(n,k-1)$ has an independent set of size at least $n/(k-1)$.
On the other hand, Ramsey's theorem asserts that for every $k$ and $\ell$ there exists a finite number $R(k,\ell)$ such that every graph with $R(k,\ell)$ vertices contains either a clique of order $k$ or an independent set of order $\ell$.
The combination of these two theorems motivated S\'os~\cite{SosRT} to introduce the so-called   Ramsey-Tur\'an theory. The typical problem asks to determine the maximum number of edges in a $K_{k+1}$-free $n$-vertex graph with independence number at most $\ell$. By Ramsey's theorem, such a graph exists only if $n < R(k+1,\ell+1)$. 
The \emph{inverse Ramsey number}  $Q(k,n)=\ell$ is the smallest possible independence number in a $K_k$-free graph on $n$ vertices, i.e., $\ell$ is the maximum integer such that $R(k,\ell)\le n$.
When investigating the asymptotic behavior of the maximum number of edges in  $K_{k+1}$-free $n$-vertex graphs with independence number at most $\ell$, usually $n$ goes to infinity, $k$ is a fixed constant, and $\ell$ depends on $n$ such that $\ell \geq Q(k,n)$.

More generally, we may require that every large subset of vertices spans not only an edge, but a clique of a given size, leading to the following definition. The \emph{$p$-independence number} of a graph $G$ is 
\[
\alpha_p(G) := \max\left\{ |U|: U \subseteq V(G) \text{ and } G[U] \text{ is } K_p\text{-free}\right\}.
\]

Notice that $\alpha_2(G)$ is the  independence number of $G$. The \emph{Ramsey-Tur\'an number}  $RT_p(n,K_q,m)$ is the maximum number of edges in an $n$-vertex $K_q$-free graph with $\alpha_p(G) \leq m$.
The \emph{Ramsey-Tur\'{a}n density} (in case it exists, which is expected) is
\[
\varrho_p(q) := \lim_{\varepsilon \to 0}\lim_{n \to \infty} \frac{RT_p(n,K_q,\varepsilon n)}{\binom{n}{2}}.
\]
These types of problems were first studied by S\'os~\cite{SosRT}, and Erd\H os and S\'os~\cite{ErdosSosRT} who determined the value of $\varrho_2(q)$ for all odd $q$. Since their work, determining Ramsey-Tur\'an densities for various $p$ and $q$ has become a classical problem in extremal combinatorics.

The case when $p=2$ and $q$ is even has proved to be significantly harder. First non-trivial upper bound was proven by Szemer\'edi~\cite{szemeredi} using his celebrated regularity lemma. In terms of lower bounds, a big breakthrough was by Bollob\'{a}s and Erd\H{o}s~\cite{BolobasErdos}, who constructed $n$-vertex $K_4$-free graphs with $\frac{n^2}{8} - o(n^2)$ edges and independence number $o(n).$ Building on these works, the case $p=2$ was fully settled by Erd\H{o}s, Hajnal, Simonovits, 
S\'{o}s and Szemer\'{e}di~\cite{MR1300968}:
\[
\varrho_2(2t+1) = \frac{t-1}{t} \quad\quad \text{\ for all \ } t \geq 1,  \quad \quad \quad \text{\ and\ } 
\quad \quad 
\varrho_2(2t) = \frac{3t-5}{3t-2} \quad\quad \text{ for all\ } t \geq 2.
\]

Furthermore, in~\cite{MR1300968}, the authors conjectured that the asymptotically extremal graphs for $\varrho_p(q)$ exhibit a certain periodic structure.
\begin{conjecture}[Conjecture \cite{MR1300968}]\label{conj:false}
The asymptotically extremal graph $G$ for $\varrho_p(q)$ has the following structure. Let $q = pt + r +2$, where $t \in \mathbb{N}$ and $r \in \mathbb{Z}_p$. Then, there exists a partition $V(G) = V_0 \cup V_1 \cup \cdots \cup V_t$ such that
\begin{itemize}
    \item $e(G[V_i]) = o(n^2)$ for all $0 \leq i \leq t$;
    \item $d_G(V_0,V_1) = \frac{r+1}{p}-o(1)$, and degrees in $G[V_0,V_1]$ differ by $o(n)$;
    \item $d_G(V_i,V_j) = 1-o(1)$ for all pairs $\{i,j\} \neq \{0,1\}$.
\end{itemize}
In particular, 
\[
\varrho_p(q) = \varrho_p^\star(q) := \frac{(t-1)(2p-r-1)+r+1}{t(2p -r -1) + r + 1}.
\]
\end{conjecture}

\begin{figure}
\begin{center}
       \begin{tikzpicture}[scale=2]
       \draw(0.5,-0.3) node[below]{$t=2$};
    \draw[line width=20 pt,color=gray!50!white,opacity=0.5] 
    (0,0) coordinate(a) -- (1,0) coordinate(b)
    (a)  -- (60:1)  coordinate(c)
    (c)--(b)
    ;
    \draw[fill=white]  
    (a) circle(0.3cm) (a) node{$V_0$}
    (b) circle(0.3cm) (b) node{$V_1$}
    (c) circle(0.3cm) (c) node{$V_2$}
    ;
    \draw
    ($(a)!0.5!(b)$) node{$\frac{r+1}{p}$}
    ($(a)!0.5!(c)$) node{$1$}
    ($(c)!0.5!(b)$) node{$1$}
    ;
    \end{tikzpicture}
\hskip 4em
       \begin{tikzpicture}[scale=2]
              \draw(0.5,-0.3) node[below]{$t=3$};
    \draw[line width=20 pt,color=gray!50!white,opacity=0.5] 
    (0,0) coordinate(a) -- (1,0) coordinate(b)
    -- (1,1)  coordinate(c)
    -- (0,1)  coordinate(d) --(a)
    ;
    \draw [line width=20 pt,color=gray!50!white,opacity=0.5] 
    (a)--(c)
    ;
    \draw [line width=20 pt,color=gray!50!white,opacity=0.5] 
    (b)--(d)
    ;
    \draw[fill=white]  
    (a) circle(0.3cm) (a) node{$V_0$}
    (b) circle(0.3cm) (b) node{$V_1$}
    (c) circle(0.3cm) (c) node{$V_2$}
    (d) circle(0.3cm) (d) node{$V_3$}
    ;
    \draw
    ($(a)!0.5!(b)$) node{$\frac{r+1}{p}$}
    ($(b)!0.5!(c)$) node{$1$}
    ($(c)!0.5!(d)$) node{$1$}
    ($(a)!0.5!(d)$) node{$1$}
    ($(a)!0.5!(c)$) node{$1$}
    ;
    \end{tikzpicture}
\end{center}
\caption{Construction from Conjecture~\ref{conj:false} for $t=2$ and $t=3$. Recall $q = pt + r +2$.}\label{fig:conjfalse}
\end{figure}
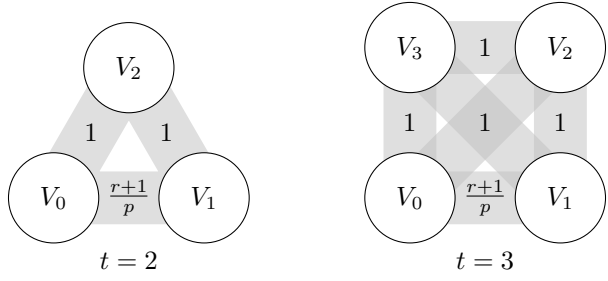

 Non-trivial lower bounds are rare
 for  $p\ge 3$, a remarkable result is due to Balogh and Lenz~\cite{BaloghLenz2} who used a product construction utilizing the so-called Bollob\'{a}s-Erd\H{o}s graph mentioned above. A recent breakthrough was  by Liu, Reiher, Sharifzadeh, and Staden~\cite{hong}, using complex high dimensional spheres. In particular, one of their constructions implies that $\varrho_3(5) = 1/6.$ Furthermore, using a different construction, they disproved Conjecture~\ref{conj:false} by showing $\varrho_{16}(22)=1/6 >5/32=  \varrho_{16}^\star(22)$. More importantly, their result implies that the  structure of the extremal graphs  is much more intricate than previously conjectured.

Our main contribution to Ramsey-Tur\'{a}n theory is proving new upper bounds as given by the following theorem.
\begin{restatable}{theorem}{thmsummaryrestate} \label{thm:summary}
    The following bounds hold:\\
     $$\varrho_4(11) \leq \frac{4}{7},\quad \quad\quad 
     \varrho_5(12) \leq \frac{10}{19},\quad\quad\quad 
    \varrho_6(12) \leq \frac{5}{12},\quad\quad\quad 
     \varrho_6(14) \leq \frac{12}{23}.$$
     This implies  $\varrho_5(12) = \frac{10}{19}.$
\end{restatable}

To prove Theorem~\ref{thm:summary} we convert the problem of determining upper bounds for $\varrho_p(q)$ into an edge-weighted Tur\'an setting, as outlined below, and then use the method of flag algebras introduced by Razborov~\cite{Razborov}.

Let us now briefly explain how the upper bounds for $\varrho_p(q)$ are obtained. First, one applies Szemer\'edi's regularity lemma to obtain an $\varepsilon$-regular equipartition of a given graph $G.$ This gives an edge-weighted cluster graph, where the weights come from the densities of the corresponding regular pairs. Recall that the clusters in the partition have $\Omega(n)$ vertices. By the assumption on  $\alpha_p(G)$, every linear subset of a cluster contains a copy of $K_p.$ This allows one to show that certain weighted configurations in the cluster graph lead to the existence of a $K_q$ in $G$. For example, to prove that $\varrho_2(4) \le 1/4,$ one shows how to embed a $K_4$ in $G$ if there is a triangle in the cluster graph with all three pairwise densities bounded away from $0$ as well as how to embed a $K_4$ in $G$ if there are two clusters with density at least roughly $1/2$. After finding such weighted configurations which lead to the existence of the desired clique, the task reduces to a weighted Tur\'{a}n problem which we discuss next. The main obstacle for obtaining further upper bounds on Ramsey-Tur\'an densities has been a lack of such weighted Tur\'an results.

\subsection{Weighted Tur\'an problems} \label{subsec:introweighted}
Given a graph $G$, an \emph{edge weighting} $w$ is a mapping $w:E(G) \to [0,1]$.
We extend $w$ to $G$ as a scaled sum of the weights of all its edges 
\[
w(G) := \frac{2}{n^2}\sum_{e\in E(G)} w(e).
\]
Note that if we consider non-edges as edges of weight $0$ then $w(G)$ is asymptotically the average weight of edges of $G$, which is $w(G) \cdot \frac{n}{n-1}$. We scale $w(G)$ by ${2}/{n^2}$ rather than $\binom{n}{2}^{-1}$ since this gives nicer extremal values.

Given a positive integer $r$, a \emph{weighted clique} is a pair $(r,f)$ where $f$ is a function $f \colon \binom{[r]}{2} \rightarrow [0, 1].$
Given a graph $G$ with an edge weighting $w,$ we say that $G$ contains a weighted clique $(r, f)$ if there is an injective function $\phi \colon [r] \rightarrow V(G)$ such that for all $1 \le i < j \le r,$ $\phi(i)\phi(j) \in E(G)$ and $w(\phi(i)\phi(j)) > f(ij).$ Given a set $\calF$ of weighted cliques,  we say $G$ is $\calF$-{\it free} if it contains no element of $\calF$.

\begin{definition}
    Given a set $\calF$ of weighted cliques, we define for $n \in \mathbb{N}$
    \[ 
    \beta_n(\calF) := \max \{ w(G):\   |V(G)| = n, \, G \text{ is } \calF\text{-free}\}. \]
    Moreover, define
    \begin{align}
    d(\calF) := \lim_{n \rightarrow \infty} \beta_n(\calF). \label{eq:df}
    \end{align}
\end{definition}

\begin{problem}\label{prob:dcalF}
    Given $\calF,$ determine $d(\calF).$
\end{problem}
Recall that $w(G) \cdot \frac{n}{n-1}$ is the average weight of edges in $G$. Using a standard averaging argument, it follows that for every $\mathcal{\calF}$, the sequence  
$\beta_n({\calF}) \cdot \frac{n}{n-1}$ is  decreasing. Hence  $\lim_{n\rightarrow \infty} \beta_n(\mathcal{F}) \cdot \frac{n}{n-1}$ exists and  so does $d(\calF).$ As noted in the previous subsection, upper bounds on $d(\calF)$ for certain families $\calF$ imply upper bounds on the Ramsey-Tur\'an densities $\varrho_p(q).$

A similar problem was investigated by F\"{u}redi and K\"{u}ndgen~\cite{FurediKungden}. Namely, for given integers $k, r$ they studied the problem  of determining the maximum weight of an integer-weighted graph where no $k$ vertices induce a subgraph of total weight more than $r$.

A special case of Problem~\ref{prob:dcalF} can be stated as follows and it is a generalization of Tur\'a{n}'s theorem. We call a function $w$ from integers $2,3,\ldots$ to $[0,1]$ a \emph{clique weighting}. 
Let $G$ be a graph,  $e$ be an edge of $G$, and $r$ be the order of a largest clique containing $e$. 
Then,  we define
\[
w(e) := w(r).
\]

The general problem   for a fixed $w$ is  to maximize $w(G)$ over all $n$-vertex graphs $G$, or over some  family of graphs $G$.

\begin{problem}
Let $w$ be a clique weighting. Find $\alpha(n)$  for a fixed $n$,  or at least  determine its limit when $n$ goes to infinity, where
\[
\alpha(n) := \max_{G:|V(G)|=n} w(G) \quad\quad {\text and}\quad\quad \alpha:=\limsup_{n\to \infty}{\alpha(n)}.
\]
\end{problem}
Since there are only finitely many $n$-vertex graphs, up to isomorphism, the maximum is well-defined.
A particularly interesting clique weighting is
\[
w_T(r) := \frac{r}{2(r-1)},
\]
which we call \emph{Tur\'an weights}.
%
Notice that Tur\'an weights are defined such that
for every $k \geq 2$, we get
\[
\lim_{n \to \infty} w_T(T(n,k)) = \frac{1}{2},
\]
and $w_T(T(n,k)) = \frac{1}{2}$, whenever $n$ is divisible by $k$.
We prove that Tur\'an graphs are exactly the graphs maximizing $w_T(G)$ among all graphs $G$ on $n$ vertices as $n$ goes to infinity. 
\begin{theorem}\label{thm:standard_weighting}   
For every graph $G$ we have
\[
w_T(G) \leq \frac{1}{2}.
\]
\end{theorem}

Theorem~\ref{thm:standard_weighting} was used by Balogh, Chen, McCourt and Murley~\cite{BaloghRamseyTuran} to obtain upper bounds for certain Ramsey-Tur\'an problems.
During the preparation of this work, Theorem~\ref{thm:standard_weighting} was independently proven by Malec and Tompkins \cite{malec2022localized}.

We also prove a stability version of Theorem~\ref{thm:standard_weighting}.
\begin{theorem} \label{thm:standard_weighting_stability}
For every $\varepsilon > 0,$ there  exist $n_0$ and $\delta > 0$ such that if $w_T(G)  \geq 1/2 -\delta$ and $n \geq n_0$,
then $G$ is in edit distance at most $\varepsilon n^2$ from some Tur\'an graph.
\end{theorem}

The rest of the paper is structured as follows. In Section~\ref{sec:proofs}, we prove our results on Tur\'{a}n weights and discuss one application. In Section~\ref{sec:otherw}, we prove some results on the general weighted Tur\'{a}n problem which might be of independent interest. Finally, we prove Theorem~\ref{thm:summary} in Section~\ref{sec:ramtur}.

\section{Proofs of Theorems~\ref{thm:standard_weighting} and~\ref{thm:standard_weighting_stability}}
\label{sec:proofs}

Our proofs are based on the method of graph Lagrangians, first introduced by Motzkin and Straus \cite{motzkin1965maxima}. Following \cite{hong}, given an $m \times m$ matrix $A,$ we define 
\[ g(A) \coloneqq \max \left\{ \bu^{\intercal} A \bu \, \big\vert \, \bu = (u_1, \dots, u_m)^\intercal,\ \  \sum_{i=1}^m u_i = 1,\  \ u_i \ge 0 \right\}. \]
The maximum is attained since it is taken over a compact set. A vector attaining the maximum is  \emph{optimal} for $A$. We say that $A$ is \emph{dense} if it has zero diagonal and for every $i \in [m],$ the matrix $A'$ obtained by removing its $i^{th}$ row and column satisfies $g(A') < g(A).$ We use the following lemma from \cite{hong} which lists useful properties of optimal vectors of a dense matrix.

\begin{lemma}[Liu, Reiher, Sharifzadeh, and Staden \cite{hong}] \label{lem:lagrangians}
    Let $m \in \mathbb{N}$, $A = (a_{ij})$ be a dense symmetric $m \times m$ matrix with nonnegative entries, and $\bu$ be optimal for $A.$ Then,
    \begin{enumerate}[label=\alph*)]
        \item $A$ is positive, that is, $a_{ij} > 0$ for every $1 \le i < j \le m,$
        \item $u_i > 0$ for every $i \in [m],$
        \item $\sum_{i\in [m] \setminus \{j\}} a_{ij} u_i = g(A),$ for every $j \in [m].$ \label{eq:degrees}
    \end{enumerate}
\end{lemma}
We note that in the original formulation of \cite{hong}, the entries of $A$ were restricted to the set $\{0, 1, \dots, p\}$ for some positive integer $p.$ However, the proof in \cite{hong} gives this stronger statement as well.

\begin{proof}[Proof of Theorem~\ref{thm:standard_weighting}]
    Assume that $G$ has at least one edge, otherwise the statement is trivial. Label the vertices of $G$ by $\{v_1, \dots, v_n\}$. Define a symmetric  $n \times n$ matrix $A = (a_{ij})$ such that $a_{ij} = w_T(v_iv_j)$ if $v_iv_j \in E(G)$ and $a_{ij} = 0$ otherwise. 
    Let $A'$ be a minimal, by inclusion,  principal submatrix of $A$  maximizing $g(A')$. 
    Since $A'$ is minimal by inclusion, it is dense. It is also nonempty, since $g$ is $0$ on the empty matrix. Let $\bu$ be optimal for $A'.$ Let $K \subseteq [n]$ be the set of indices of vertices corresponding to the rows (and columns) of $A'$ and write $k = |K|.$ By Lemma~\ref{lem:lagrangians}, $A'$ is positive, that is, the vertices indexed by $K$ form a clique in $G.$ Since $w_T(r)$ is decreasing in $r,$ it follows that $a_{ij} \le w_T(k)$ for all $i, j \in K.$ Hence,
    \[ g(A) \le g(A') \leq \sum_{i \in K} u_i \sum_{j \in K, j \neq i} u_j w_T(k) = w_T(k) \sum_{i \in K} u_i (1 - u_i) = w_T(k) \left(1 - \sum_{i \in K} u_i^2\right) \le w_T(k)  \left(1 - \frac1k \right) = \frac12, \]
    where we used that $\sum_{i \in K} u_i = 1$ and in the last inequality we used Jensen's inequality. On the other hand, by using $\bx = (1/n, \dots, 1/n),$ we obtain
    \[ w_T(G) = \frac{2}{n^2} \sum_{e \in E(G)} w_T(e) = 
    \bx^\intercal A \bx \le g(A) \le \frac12, \]
    completing the proof.
\end{proof}

Next, we prove the stability version as given by Theorem~\ref{thm:standard_weighting_stability}. Our proof is based on a careful analysis of the argument of \cite{malec2022localized}.
\begin{proof}[Proof of Theorem~\ref{thm:standard_weighting_stability}]
    Fix $\varepsilon > 0$ and let $0 < \delta \ll \varepsilon$ be a small constant to be chosen implicitly later. Additionally, let $\delta' = 10 \delta^{1/3}.$ Let $G$ be an $n$-vertex graph with sufficiently large $n$ satisfying $w_T(G) \ge 1/2 - \delta.$ We write  $V = V(G).$ For $U \subseteq V,$ define $W(U) := \sum_{e \in E(G[U])} w_T(e)$ and note that $W(V) = \frac{n^2}{2} w_T(G).$ For disjoint sets $A, B \subseteq V(G),$ define $W(A, B) := \sum_{e \in E(G[A,B])} w_T(e).$ 

     Define a sequence of graphs $G_1, G_2, \dots$ as follows: set $G_1 = G$; for each  $i \ge 1,$ let $C_i$ be a clique of maximum size in $G_i$ and let $G_{i+1} := G_i \setminus C_i.$ Let $t \ge 1$ be the minimum index such that $|V(G_{t+1})| \le (1 - \delta') n.$ Let $A = \bigcup_{i=1}^t V(C_i)$ and $B = V(G) \setminus A.$ 
    
    If $|B| \le (\delta')^{1/3} n,$ then  $t = 1$. This implies that $C_1$ is a clique in $G$ of size at least $(1 - (\delta')^{1/3})n$, and the statement trivially holds. Hence, from now on  we may  assume $|B| > (\delta')^{1/3}n.$ 
    We also have $|A| \geq \delta'n$ by the definition. 
    Together they imply 
    $$|A||B| > \min\{(\delta')^{1/3} (1 - (\delta')^{1/3})n^2, \delta'(1 - \delta')n^2\} > 0.99\delta' n^2.$$

    We use the following:
    \[ 
    \frac{n^2}{2}w_T(G)=W(V) = W(A) + W(B) + W(A, B).
    \]
    
     By Theorem~\ref{thm:standard_weighting}, we have that $W(A) \le |A|^2 / 4$ and $W(B) \le |B|^2 / 4.$ Using $|A| + |B| = n$ and our initial assumption that $w_T(G) \geq 1/2-\delta$, we have
    \begin{align}
        W(A, B) &= W(V) - W(A) - W(B) \ge \frac{1}{4} \left( (1 - 2\delta)n^2 - |A|^2 - |B|^2 \right) = \frac{1}{4} \big( 2(1-2\delta) |A||B| - 2\delta|A|^2 - 2\delta|B|^2\big)\nonumber\\
        &\ge \left(\frac{1}{2} - \delta - \frac{\delta n^2}{|A||B|}\right)|A||B| \ge \left(\frac{1}{2} - \frac{1}{4}\delta'^2 \right) |A||B|,\label{eq:WAB-lb}
    \end{align}
    where in the last relation we used $\delta' = 10\delta^{1/3}$ and $|A||B| \ge 0.99 \delta' n^2.$
    
    Fix $i \le t, v \in B$ and let $N = N(v) \cap C_i.$ Note that $N \cup \{v\}$ is a clique.  Denote $|N|$ by $e(v, C_i)$. Hence, $N \neq C_i$ as otherwise $C_i \cup \{v\}$ is a  clique larger than $C_i$ in $G_i.$ Each edge from $v$ to $C_i$ is in the clique $N \cup \{v\}$, so its weight is at most $w_T(|N| + 1).$
    Hence, 
    \[ W(\{v\}, C_i) \le |N| w_T(|N| + 1) = \frac{|N| + 1}{2} = \frac{e(v, C_i) + 1}{2} \le \frac{|C_i|}2. \]
    
    Now, we can obtain an upper bound on $W(A, B)$:
    \begin{equation} \label{eq:AB-general}
        W(A, B) = \sum_{i=1}^t \sum_{v \in B} W(C_i, \{v\}) \le \sum_{i=1}^t \sum_{v \in B} \frac{1}{2}(e(v, C_i) + 1) \le \sum_{i=1}^t \sum_{v \in B} \frac{1}{2} |C_i| =  \frac{|A||B|}{2}.
    \end{equation} 
    
    Assume that for all $i \in [t],$ there are at least $\delta' n$ vertices $v \in B$ such that $e(v, C_i) + 1 \leq (1 - \delta')|C_i|.$ Then, we can obtain a stronger bound on $W(A, B)$ as follows:
    \begin{align*}
        W(A, B) &= \sum_{i=1}^t \sum_{v \in B} W(C_i, \{v\}) \leq \sum_{i=1}^t \sum_{v \in B} \frac{1}{2}(e(v, C_i) + 1) \le \sum_{i=1}^t \frac{1}{2}\big((|B| - \delta'n) |C_i| + \delta' n (1 - \delta') |C_i|\big)\\
        &= \sum_{i=1}^t \frac{1}{2}|C_i| (|B| - \delta'^2 n) \le \frac{1}{2}\left(1 - \delta'^2\right)|A||B|,
    \end{align*}
    a contradiction with \eqref{eq:WAB-lb}.
    
    Hence, we can fix an $i \in [t]$  such that for all but at most $\delta' n$ vertices $v \in B$ it holds that $e(v, C_i)+1   > (1 - \delta') |C_i|.$ Let $B' \subseteq B$ denote the set of such vertices, so $|B'| \ge |B| - \delta'n \ge ((\delta')^{1/3} - \delta')n.$
    
    Assume first that $|A| \ge 2 \delta' n.$ Then, $i = t = 1$ and $A = C_1$ is a clique. Let $uv \in G[B']$ be an arbitrary edge. Note that
    \[
        |N(u) \cap N(v) \cap A| \ge (1 - 2\delta')|A|-2,
    \]
    hence 
    \[
    w_T(uv) \le w_T((1 - 2\delta')|A| ) < \frac{1}{2} + \frac{2}{|A|} < \frac{1}{2} + \frac{1}{\delta' n}.
        \]
    Since there are at most $\delta' n^2$ edges touching $B \setminus B',$ using Theorem~\ref{thm:standard_weighting} and \eqref{eq:AB-general}, we have
    \begin{align*}
        \frac{n^2}{2} \left(\frac{1}{2} - \delta\right) &\le W(V) \le W(A) + W(A,B') + W(B') + \delta' n^2 \le \frac{|A|^2}{ 4} + \frac{|A|(n - |A|)}{ 2} + W(B') + \delta' n^2 \\
        &= W(B') + \frac{n|A|}{ 2} - \frac{|A|^2}{ 4} + \delta' n^2,
    \end{align*}
    which implies 
    \[ 
    W(B') \ge \frac{n^2}{ 4} - \frac{n|A|}{ 2} + \frac{|A|^2}{ 4} - 2 \delta' n^2 = \frac{(n - |A|)^2}{4} - 2 \delta' n^2 \ge \left(\frac{1}{4} - 4 (\delta')^{1/3}\right) |B'|^2, 
    \]
    where the last inequality holds because $|B'| \ge ((\delta')^{1/3} - \delta') n.$ On the other hand, $W(B') \le e(G[B']) \cdot (1/2 + 1 / (\delta' n)),$ so 
    \[ 
    e(G[B']) \ge \frac{(1/4 - 4 (\delta')^{1/3}) |B'|^2}{1/2 + 1 / (\delta'n)} \ge \frac{1 - \varepsilon}{2} \cdot |B'|^2,
    \]
    where we chose $\delta$ to be small enough with respect to $\varepsilon.$ Hence, 
    \[
    e(G) \ge \binom{n}{2} - \frac{\varepsilon}{2} |B'|^2 - 2\delta' n^2 \ge (1 - \varepsilon) \binom{n}{2},
    \] 
    so $G$ has  edit distance at most $\varepsilon$ to the complete graph and we are done.
     
    Therefore, we may assume that $|A| < 2 \delta' n,$ so $|B| \ge (1 - 2\delta') n$ and $|B'| \ge (1 - 3\delta')n$.
    There are at most $3\delta' n^2$ edges touching $V \setminus B',$ so
    \[ W(B') \ge W(V) - 3\delta' n^2 > \frac{1}{4}(1 - 20\delta')n^2.\]
    Define $k := |C_t|,$ and recall that $C_t$ is the last clique removed from $G.$ There are two cases:\\
    
    \textit{Case 1: $k < 1 / \delta'.$} Then, for each $v \in B'$ we have $e(\{v\}, C_t) + 1 \ge (1 - \delta')k > k - 1.$ That is, each $v \in B'$ is adjacent to all but one vertex in $C_t.$ Hence, we can partition $B'$ into sets $B'_1, B'_2, \dots, B'_k$ according to which vertex of $C_t$ is missing in the neighbourhood of $v \in B'.$ Note that for every $j \in [k],$ $B'_j$ is an independent set. Indeed, otherwise taking the two vertices which form an edge in $B'_j$ and their $k-1$ common neighbours in $C_t,$ we obtain a clique of size $k+1$ in $G_t,$ a contradiction. Now, consider an edge $uv \in G[B'].$ Note that $|N(u) \cap N(v) \cap C_t| = k-2,$ so $uv$ lies in a clique of size $k.$ Hence, $w_T(uv) \le w_T(k).$ 
    Therefore,  $G[B']$ is a $k$-partite graph with 
    \[ e(G[B']) \ge \frac{W(B')}{w_T(k)} \ge \left(1 - 20\delta'\right) \frac{k-1}{k} \frac{n^2}{2}. \]
    By the standard stability result for $K_{k+1}$-free graphs \cite{erdsim}, it follows that $G[B']$ is in edit distance at most $\varepsilon / 2 n^2$ from $T(|B'|, k).$ Since $|B'| \ge (1 - 3\delta')n,$ the result follows.
    
    \textit{Case 2: $k \ge 1 / \delta'.$} Choose $u, v \in B'$ such that $uv \in E(G).$ Note that $|N(u) \cap N(v) \cap C_t| \ge (1 - 2\delta')k,$ hence $w_T(uv) \le w_T((1 - 2\delta')k + 2) < 1/2 + 1/k.$
    Hence, 
    $$e(G[B']) \ge \frac{W(B')}{1/2 + 1/k} \ge \frac{1 - 20\delta'}{1 + 2/k} \cdot \frac{n^2}{2} \ge (1 - \varepsilon/2)\frac{n^2}{2},$$ where we chose $\delta'$ to be sufficiently small compared to $\varepsilon.$ Again, since $|B'| \ge (1 - 3\delta')n,$ 
    $G$ is $\varepsilon$-close to the complete graph and the statement follows.
\end{proof}

To capture the extremal examples, we compare rescaled weights. Let $w$ be a clique weighting. We define a rescaling $t_w$ as  
\[
t_w(r) := \frac{2(r-1)}{r} w(r).
\]
Note that $t_{w_T}(r) = 1$ for all $r \in \mathbb{N}$.
Observe that 
\[
t_{w}(r) =  2\cdot \lim_{n\to\infty}w(T(n,r)),
\]
i.e. it is the same as twice the weights of Tur\'an graphs.  

\begin{corollary}\label{conj:w}
If $w$ is a clique weighting, where $1 = \max_r t_w(r)$, then for every $n$-vertex graph $G,$
\[
w(G) \leq \frac{1}{2}.
\]

Moreover, for every $\varepsilon_1, \varepsilon_2 > 0,$ there exist $n_0$ and $\delta > 0$ such that if 
for each $r$ either $t_w(r) = 1$ or $t_w(r) < 1-\varepsilon_2$,  
$w(G) \geq 1/2-\delta$, and $n \geq n_0$,
then $G$ is in edit distance at most $\varepsilon_1 n^2$ from some Tur\'an graph $T(n,r)$, where $r$ satisfies  $t_w(r)=1$.
\end{corollary}


%
%

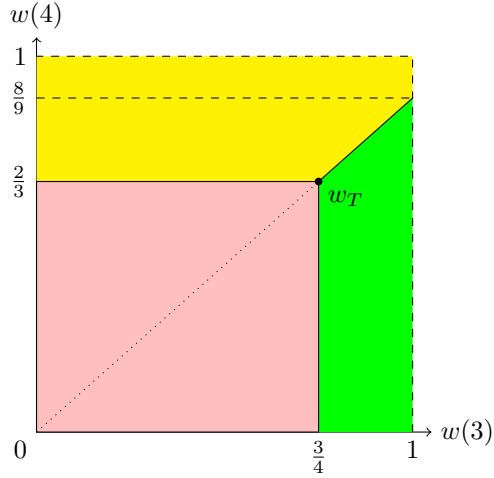
\begin{figure}
\begin{center}
    \begin{tikzpicture}[scale=2.5]
    \draw[->](0,0) -- (2.1,0) node[right]{$w(3)$};
    \draw[->](0,0) -- (0,2.1)node[above]{$w(4)$};;
    \draw(3/2,0) -- (3/2,4/3) -- (0,2/3);
    \draw (2,0) node[below]{$1$};
    \draw (0,2) node[left]{$1$};
    \draw (3/2,0) node[below]{$\frac{3}{4}$};
    \draw (0,4/3) node[left]{$\frac{2}{3}$};
    \draw (0,16/9) node[left]{$\frac{8}{9}$};
    \draw (3/2,4/3) -- (1,16/9);
    \draw (0,0) node[below left]{0};
    \draw[fill=pink] (0,0) rectangle (3/2,4/3);
    \fill[yellow] (0,4/3) -- (3/2,4/3) -- (2,16/9) -- (2,2) -- (0,2) -- cycle;
    \fill[green]  (3/2,0) -- (3/2,4/3) -- (2,16/9) -- (2,0) -- cycle; 
    \draw[dashed] (0,16/9) -- (2,16/9);
    \draw[fill = black](3/2,4/3) circle(0.5pt) node[below right]{$w_T$};
    \draw(3/2,0)--(3/2,4/3) -- (0,4/3) (3/2,4/3) -- (2,16/9);
    \draw[dotted] (0,0) -- (3/2,4/3);
    \draw[dashed](2,0)--(2,2)--(0,2);
    \end{tikzpicture}
\end{center}
\caption{Case with $w(2) = 1 = w_T(2)$ and point $w_T$ depicted.}\label{fig:w}
\end{figure}

\begin{proof}
Let $w$ be a clique weighting with $1=\max_rt_w(r)$.
Observe that $w(r) \leq w_T(r)$ for all $r$ as $t_w(r) \leq 1$.
Hence, $w(G) \leq w_T(G) \leq 1/2$ by Theorem~\ref{thm:standard_weighting}. If $w(G) \geq 1/2-\delta$ then $G$ is in edit distance $o(n^2)$ from a Tur\'an graph $T(n,r)$ for some $r$, by Theorem~\ref{thm:standard_weighting_stability}. So, we can write $E(G) = (E(T(n, r)) \setminus F) \cup F',$ where $|F|, |F'| = o(n^2).$ Let $G' = T(n, r) \setminus F,$ then $w(G') \ge 1/2 - \delta - o(1).$ Note that each edge in $T(n, r)$ is in $\Omega(n^{r-2})$ $r$-cliques and removing an edge from $T(n,r)$ removes $O(n^{r-2})$ $r$-cliques, so $G'$ has $o(n^r)$ fewer $r$-cliques than $T(n, r).$ It follows that all but $o(n^2)$ edges of $G'$ are contained in an $r$-clique, implying $w(r) = w_T(r)$, as claimed.
\end{proof}

Now we demonstrate Corollary~\ref{conj:w} on a small example of $K_5$-free graphs. Let $w$ be a clique-weighting, where $w(2)=1=w_T(2)$,
see Figure~\ref{fig:w} for guidance. Corollary~\ref{conj:w} implies the following (where we chose these particular numbers as they were used in~\cite{BaloghRamseyTuran}):
\begin{itemize}
\item If $w(3) \leq 3/4$ and $w(4) \leq 2/3$, then $T(n,2)$ is extremal.
\item If $w(3) \geq 3/4$ and $w(3) \geq \frac98 w(4)$, then $T(n,3)$ is extremal.
\item If $w(4) \geq 2/3$ and $w(3) \leq \frac98 w(4)$, then $T(n,4)$ is extremal.
\end{itemize}

\section{Other weights}\label{sec:otherw}
Recall the definitions of a weighted clique and of $d(\calF)$ from the beginning of
Section~\ref{subsec:introweighted}. For $a \in [0,1]$ we use $K_r^a$ to denote the weighted clique $(r, f)$ where $f(ij) = a$ for every  $ij \in \binom{[r]}{2}.$
Approaching Problem~\ref{prob:dcalF} of determining $d(\calF)$ for general $\calF$ seems to be hard. However, the following lemma proves that if $\calF$ contains $K_{r+1}^0$ (i.e. $\calF$-free graphs  are required to be $K_{r+1}$-free), then the asymptotically extremal graphs for $\beta_n$ are blow-ups of a $t$-clique for some $t \le r.$ We characterize these blow-ups as follows.

\begin{definition}
    Given a positive integer $t,$ a weighted clique $(t, f)$ and a $t$-tuple $\bx=(x_1, \dots, x_t)$ of positive real numbers satisfying $\sum_{i=1}^t x_i = 1,$ we define the \emph{$(n, t, f, \bx)$-blow-up} as the following weighted graph $G.$ Let $V(G) = V_1 \dot\cup V_2\cup  \dots \dot\cup V_t$,  where $|V_i| = \lfloor x_i n\rfloor$ for $1 \le i \le t-1$ and $|V_t| = n - \sum_{i=1}^{t-1} |V_i|.$ Let $G$ be the complete $t$-partite graph with parts $V_1, \dots, V_t$, where the weights of the edges between $V_i$ and $V_j$ are equal to $f(ij).$
\end{definition}

\begin{lemma} \label{lemma:main}
    Let $\calF$ be a finite set of weighted cliques containing $K_{r+1}^0.$ Then, there exists a weighted clique $(t, f)$ with $t \le r$  and a $t$-tuple $\bx=(x_1, \dots, x_t)$ of positive real numbers satisfying 
    \begin{enumerate}[label=\alph*)]
        \item $\sum_{i=1}^t x_i = 1,$
        \item \label{weighted:degrees} $\sum_{j \neq i} x_j f(ij) = d(\calF)$   for every $  i \in [t],$
    \end{enumerate}
    such that the $(n, t, f, \bx)$-blow-up $G_n$ is $\calF$-free and satisfies $w(G_n) = (1 - o(1)) d(\calF).$
\end{lemma}

\begin{proof}
    Let $(G, w)$ be a weighted $\calF$-free graph on  vertex set $[n]$ with $w(G) = \beta_n(\calF).$ Let $A = (a_{ij})$ be the $n\times n$ symmetric matrix such that $a_{ij} = w(ij)$ and $a_{ii} =0$. As in the proof of Theorem~\ref{thm:standard_weighting}, we apply Lemma~\ref{lem:lagrangians} to obtain a clique $K$ on $\{v_1, \dots, v_t\}$ with $t \le r$ and positive real numbers $x_1,\dots, x_t$ with $\sum_{i=1}^t x_i = 1$ and $\sum_{j \in [t], j \neq i} x_j w(v_iv_j) = g(A)$ for all $i \in [t].$ Then, 
    \[ \beta_n(\calF) = w(G) = \frac{2}{n^2} \sum_{e \in E(G)} w(e)  \le g(A). \]
    On the other hand, for $\bx = (x_1, \dots, x_t)$ and $f \colon {[t] \choose 2} \rightarrow [0,1]$ defined by $f(ij) = w(v_iv_j),$ the $(n, t, f, \bx)$-blowup $G_n$ is $\calF$-free because $G$ is $\calF$-free and satisfies
    \[ w(G_n) \sim \sum_{1 \le i < j \le t} w(v_iv_j) (x_i n)(x_j n) \cdot \frac{2}{n^2} = \sum_{1 \le i \neq j \le t} w(v_iv_j) x_i x_j = g(A). \]

    Therefore, for each $n \in \mathbb{N},$ we obtain parameters $t(n) \le r,$ and $\bx(n) \in \{0, 1\}^{t(n)}$ as above. Hence, there exist $t \le r$ and an infinite sequence $a_n$ of positive integers such that $t(a_n) = t$ and $\lim_{n \rightarrow \infty} \bx(a_n) = \bx,$ which implies the statement.
\end{proof}

A triangle is \emph{$(a,i)$-heavy} if at least $i$ of its edges have weight strictly more than $a$. The following theorem determines the extremal $K_4$-free graphs with no $(a, 3)$-heavy triangles.
\begin{theorem}\label{thm:heavy}
Let $0 < a \le 1$ and $\calF = \{ K_3^a, K_4^0\}$, then  $d(\calF) = \frac{2}{4-a}.$
\end{theorem}
\begin{proof}
    Applying Lemma~\ref{lemma:main} with $r=3$ gives us a weighted clique $(t, f)$ and a $t$-tuple $\bx$ such that the $(n, t, f, \bx)$-blowup $G_n$ is asymptotically optimal. If $t = 2,$ then $w(G_n) = 1/2.$ If $t = 3,$ then one of the weights  is at most $a$. Without loss of generality, assume $f(23) \le a.$ Clearly, $w(G_n)$ is maximized when $f(12) = f(13) = 1$ and $f(23) = a.$ Then we have $d(\calF) = x_2 + x_3 = x_1 + a x_3 = x_1 + ax_2$ and $x_1 + x_2 + x_3 = 1.$ Solving this system of equations, we obtain $x_1 = \frac{2-a}{4-a}, x_2 = x_3 = \frac{1}{4-a}$ and $d(\calF) = \frac{2}{4-a} > 1/2,$ as needed.
    See Figure~\ref{fig:11a} for a sketch of the construction.
    
\end{proof}

Taking $a=3/4$ we obtain the following corollary which was used in \cite{BaloghRamseyTuran} to obtain several Ramsey-Tur\'{a}n type of results.
\begin{corollary} 
Let $G$ be an $n$-vertex $K_4$-free graph with a weight function $w : E(G) \to (0,1]$. If $G$ contains no $(3/4, 3)$-heavy triangles, then $w(G) \leq \frac{8}{13} + o(1).$
\end{corollary}

%
%

\begin{figure}
\begin{center}
       \begin{tikzpicture}[scale=2]
    \draw[line width=20 pt,color=gray!50!white] 
    (0,0) coordinate(a) -- (1,0) coordinate(b)
    (a)  -- (60:1)  coordinate(c)
    (c)--(b)
    ;
    \draw[fill=white]  
    (a) circle(0.3cm) (a) node{$\frac{n}{4-a}$}
    (b) circle(0.3cm) (b) node{$\frac{n}{4-a}$}
    (c) circle(0.3cm) (c) node{$\frac{(2-a)n}{4-a}$}
    ;
    \draw
    ($(a)!0.5!(b)$) node{$a$}
    ($(a)!0.5!(c)$) node{$1$}
    ($(c)!0.5!(b)$) node{$1$}
    ;
    \end{tikzpicture}
\end{center}
\caption{The extremal weighted triangle from Theorem~\ref{thm:heavy}.}\label{fig:11a}
\end{figure}
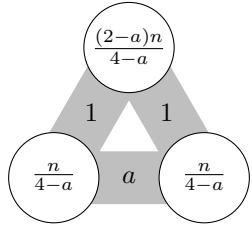

A $K_q$ is $a$-{\it chubby}, if it has an edge of weight larger than $a$. For $r \ge q,$ we determine the maximum weight of a $K_{r+1}$-free graph with no chubby $K_q.$ In particular, this determines the extremal $K_{r+1}$-free graphs with no $(a, 1)$-heavy triangles. The possible extremal graphs are the $r$-clique with all edges of weight $a$ and the $(q-1)$-clique with all edges of weight $1$.

\begin{theorem}
    Given integers $2 \le q \le r$ and $0 \le a \le 1$, let $f \colon \binom{[q]}{2}\to[0,1]$ be defined as $f(12) = a$ and $f(ij) = 0$ for all $\{i,j\} \ne \{1, 2\}.$ Denoting $\calF = \{ (q, f), K_{r+1}^0\},$ we have 
    \[ d(\calF) = \max\{ 1 - 1/(q-1), a(1 - 1/r) \}. \] 
\end{theorem}
\begin{proof}
    Let $t, f, \bx$ be the values obtained by applying Lemma~\ref{lemma:main} to the family $\calF$ such that the $(n, t, f, \bx)$-blowup $G_n$ satisfies $w(G_n) = (1 - o(1))d(\calF).$ 
    
    If $t \ge q,$ then, since every edge of $G_n$ is in a $q$-clique, we have that $f(ij) \le a$ for all $ij \in \binom{[t]}{2}$ and $w(G_n)$ is maximized if equality holds for all $i, j$ and $t = r.$ In that case, $w(G_n) \sim a (1 - 1/r).$
    
    On the other hand, if $t < q,$ then $w(G_n)$ is maximized by taking $t = q-1$ and $f(ij) = 1$ for all $ ij \in \binom{[t]}{2},$ so in this case $w(G_n) \sim 1 - 1 / (q-1).$ Taking the maximum of the above two cases finishes the proof.
\end{proof}

Finally, we consider the problem of forbidding an $(r+1)$-clique and an $(a, 2)$-heavy triangle. Here the extremal examples are the complete bipartite graph with all edges of weight $1$ and a blowup of an $r$-clique in which the edges of weight $1$ form a matching of size $\lfloor r/2 \rfloor$ and the other edges have weight $a.$

\begin{theorem}
    Fix an integer $r$ and a real $0  \le a \le 1.$ Define  $h \colon \binom{[3]}{2}\to [0,1]$ as $h(12) = h(13) = a, h(23) = 0.$ Denoting $\calF = \{ (3, h), K_{r+1}^0\},$ we have 
    \[ d(\calF) = \begin{cases}
        \frac{1}{2}, &\text{if } a \le \frac{1}{2},\\
        a + \frac{1-2a}{r}, &\text{if } a > \frac{1}{2} \text{ and } r \text{ is even},\\
        \frac{a^2(r-1)}{a(r+1) - 1}, &\text{if } a > \frac{1}{2} \text{ and } r \text{ is odd}.
    \end{cases}
    \]
\end{theorem}
\begin{proof}
    Let $(n, t, f, \bx)$ be the values obtained by applying Lemma~\ref{lemma:main} to $\calF$ and let $G_n$ denote the $(n, t, f, \bx)$-blow-up. Note that we may assume that all edge weights $f(ij)$ are either $a$ or $1,$ as otherwise we could increase the weight while preserving the property that the $(n, t, f, \bx)$-blow-up is $\calF$-free.
    
    Since $G_n$ is $\calF$-free, the edges $ij$ with $f(ij) = 1$ form a matching. The weight of $G_n$ is clearly maximized if this matching is of largest possible size. Without loss of generality, assume that $f(2\ell-1,2\ell) = 1$ for all $1 \le \ell \le \lfloor t/2 \rfloor,$ and $f(ij) = a$ for all other pairs. By part~\ref{weighted:degrees} in Lemma~\ref{lemma:main}, it follows that $x_{2\ell-1} = x_{2\ell}$ for all $\ell \in [ \lfloor t/2 \rfloor].$
    
    First, assume that $t$ is even. Then, for all $i \in [t]$,
    \begin{equation} \label{eq:t-even}
        d(\calF) = \sum_{j \neq i} x_j f(ij) = x_i + a \cdot (1 - 2x_i) = a + (1-2a) x_i. 
    \end{equation}
    If $a = 1/2,$ then $d(\calF) = 1/2.$ Otherwise, \eqref{eq:t-even} implies that $x_1 = x_2 = \ldots = x_t = 1/t,$ so $d(\calF) = a + \frac{1-2a}{t}.$
    
    Now, assume that $t$ is odd. From part~\ref{weighted:degrees} of Lemma~\ref{lemma:main}, we obtain for all $i \in [t-1]$
    \[ 
    d(\calF) = x_i + a (1 - 2x_i) = a + (1 - 2a) x_i, 
    \]
    If $a = 1/2,$ it follows that $d(\calF) = 1/2.$ Otherwise, we obtain $x_1 = x_2 =
    \ldots = x_{t-1}$ and so $x_t = 1 - (t-1) x_1,$ which implies
    \[ d(\calF) = a (1 - x_t) = a(t-1)x_1 \le a. \]
    If $a < 1/2,$ then we have  already seen that $d(\calF) \ge 1/2,$ so assume from now on that $a \ge 1/2.$
    Combining the two equations, we obtain $x_1 = \frac{a}{a(t+1) - 1},$ which is well-defined since $a(t+1) - 1 \ge 1/2 \cdot 3 - 1 > 0.$ We conclude that in this case
    \[ d(\calF) = a(t-1)x_1 = \frac{a^2(t-1)}{a(t+1) - 1}. \]
    
    We have shown that $d(\calF) = 1/2$ if $a \le 1/2.$ What is left to verify is that if $a > 1/2,$ then $d(\calF)$ is maximized by taking $t = r.$ It is enough to show that for any $a > 1/2$ and any even $t \ge 2,$ we have
    \[ a + \frac{1-2a}{t} \le \frac{a^2t}{a(t+2) - 1} \le a + \frac{1 - 2a}{t+2}. \]
    This is a straightforward calculation so we omit the details.
\end{proof}

\section{Ramsey-Tur\'an density}~\label{sec:ramtur}
Recall the definition of Ramsey-Tur\'{a}n density:
\[
\varrho_p(q) = \lim_{\varepsilon \to 0}\lim_{n \to \infty} \frac{RT_p(n,K_q,\varepsilon n)}{\binom{n}{2}},
\]
where $RT_p(n, K_q, m)$ is the maximum number of edges in an $n$-vertex $K_q$-free graph in which every $m+1$ vertices contain a clique of size $p.$

Given that the case $p=2$ is fully resolved, we will focus on the case when $p\ge 3$. It is not hard to see that $\varrho_p(q)=0$ when $q\le p+1$. 
We present a table of what is known about $\varrho_p(q)$, when both $p$ and $q$ are small. Note that if the lower and upper bounds are not matching, it is not clear which one is closer to the truth.
\begin{center}
\begin{tabular}{c|c|c|c|c|c|c|c|c|c|c}
\diaghead(-1,1){aa}{$p$}{$q$}
&      5 &  6 &  7 &  8 & 9 & 10 & 11 & 12  & 13 & 14   \\ \hline
        3 &      H &  S &  E &  H &   &  E &  H &    &  E  & H \\
        4 &     0  & H  &  S & S & E & H  &  $\star$   &    &  E   & H \\
        5 &     0  & 0  &  S & S  & S & S  &  E  &  $\star$  &     &\\
        6 &     0  & 0  &  0 & S  & S & S  &  S  &  $\star$  &  E  & $\star$\\
\end{tabular}
\end{center}

\begin{itemize}
\item[E] means  $q \equiv 1 \pmod{p}$. It was proved in \cite{MR1300968} that the optimal construction is a complete $(q-1)/p$-partite graph, with a graph with sublinear $\alpha_p$-independence number placed into each of the classes, i.e., $\varrho_p(q) = \frac{q-1-p}{q-1}.$

\item[H] means Theorem 1.4 in~\cite{hong}: 
\end{itemize}

\[
\varrho_3(3t+2)= \frac{5t-4}{5t+1} \quad \quad \text{ and } \quad\quad  \varrho_4(4t+2) = \frac{7t-6}{7t+1}.
\]

\begin{itemize}
\item
[S] means  $q = p+\ell$, where $\ell \leq \min\{p,5\}$.  In \cite{MR1300968}, the upper bound $\varrho_p(q) \leq \frac{\ell-1}{2p}$ was proved. For lower bounds see below.
\item
[$\star$] We improve the upper bound. 
\end{itemize}

{\bf Case $p=3$:}\\
For $q=6$, a construction from~\cite{BaloghLenz2} shows
that $\varrho_3(6) \ge  1/4$. Joining this construction with independent sets (here joining means adding an independent set of vertices and all the cross-edges) iteratively (see also Figure~\ref{fig:conjfalse}) yields that 
$\varrho_3(9) \ge  4/7$ and $\varrho_3(12) \ge  7/10$.
The upper bound $\varrho_3(9) \le  3/5$ is from~\cite{MR1300968}.

%


{\bf Case $p=4$:}\\
The lower bounds were proved both in~\cite{BaloghLenz2} and~\cite{hong}. The upper bounds for $q\le 9$ are from~\cite{MR1300968}, and we prove here $\varrho_4(11) \le 4/7$: 
$$\varrho_4(6) =  1/8,\quad \varrho_4(7) =  1/4,\quad  1/4\le \varrho_4(8) =  3/8,\quad \varrho_4(11) = 4/7.
$$


{\bf Case $p=5$:}\\
The lower bounds were proved both in~\cite{hong} for $q\le 9$ and $q=12,13$, and in~\cite{BaloghLenz2} for the others. The upper bounds for $q\le 10$ are from~\cite{MR1300968},  and we settle the case  $q=12$: 
$$\varrho_5(7) = \frac{1}{10},\quad  \varrho_5(8) = \frac  {1}{5},\quad \frac {1}{4}\le \varrho_5(9) \le  \frac {3}{10},\quad \frac {1}{4}\le\varrho_5(10) \le  \frac {2}{5}, \quad  \varrho_5(12)=\frac {10}{19},\quad \frac {5}{9} \leq \varrho_5(13),\quad \frac{4}{7}\le \varrho_5(14).$$


{\bf Case $p=6$:}\\
The upper bounds are from~\cite{MR1300968} and
 the constructions are from~\cite{hong}. The cases $q=10,11,12$ are from~\cite{BaloghLenz2}.
In this paper, we prove the upper bounds for $q=12, 14$. The construction for $q=12$ is from~\cite{BaloghLenz2}, and for $q=14$ it is from~\cite{hong}: 
$$\varrho_6(8) =  \frac {1}{12},\quad  \varrho_6(9) =  \frac{1}{6},\quad \varrho_6(10) =  \frac{1}{4},\quad \frac{1}{4}\le \varrho_6(11) \le  \frac{1}{3},\quad \frac{1}{4}\le \varrho_6(12) \le  \frac{5}{12},
\quad \varrho_6(14) =  \frac{12}{23}.$$


Recall that Conjecture~\ref{conj:false} was disproved by Liu, Reiher, Sharifzadeh, and Staden~\cite{hong} by showing $\varrho_{16}(22)=1/6 >5/32=  \varrho_{16}^\star(22)$. While they resolved many  cases when  $p=3$ and $p=4$,  they also asked what happens for larger $p$. While we solve only some sporadic cases, we think that it is interesting to determine, or at least upper bound some values, to see what constructions are or could be sharp. 

Below, we recollect all the new results we prove in this paper.

\thmsummaryrestate*

We tried to apply the method from the proof of Theorem~\ref{thm:summary} on   
$\varrho_3(12)$,
$\varrho_5(13)$,
$\varrho_5(14)$,
$\varrho_4(12)$, and
$\varrho_7(16)$.
However, we would need more computational power to obtain reasonable bounds.
The rest of the paper is dedicated to proving Theorem~\ref{thm:summary}.

\subsection{$\varrho_5(12) = \frac{10}{19}$}\label{sec:512}

Corollary 1.2 in~\cite{hong} gives a construction showing 
$\varrho_5(12) \geq \frac{10}{19}$, see Figure~\ref{fig:512}. 
The existence of such a construction was suggested by Conjecture~\ref{conj:false}.
In this section we show the construction is asymptotically best possible. The results in the rest of the  section are proved in a similar way and only the differences from this section will be mentioned.

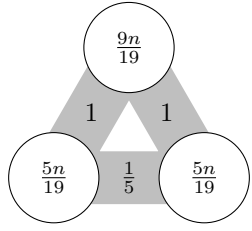
\begin{figure}
\begin{center}
       \begin{tikzpicture}[scale=2]
    \draw[line width=20 pt,color=gray!50!white] 
    (0,0) coordinate(a) -- (1,0) coordinate(b)
    (a)  -- (60:1)  coordinate(c)
    (c)--(b)
    ;
    \draw[fill=white]  
    (a) circle(0.3cm) (a) node{$\frac{5n}{19}$}
    (b) circle(0.3cm) (b) node{$\frac{5n}{19}$}
    (c) circle(0.3cm) (c) node{$\frac{9n}{19}$}
    ;
    \draw
    ($(a)!0.5!(b)$) node{$\frac{1}{5}$}
    ($(a)!0.5!(c)$) node{$1$}
    ($(c)!0.5!(b)$) node{$1$}
    ;
    \end{tikzpicture}
\end{center}
\caption{Sketch of a construction for $\varrho_5(12) \geq \frac{10}{19}$.}\label{fig:512}
\end{figure}

\begin{proof}[Proof outline.]
We only need to prove $\varrho_5(12) \leq \frac{10}{19}$.
As mentioned in Section~\ref{sec-intro-ramsey-turan},
we use the method developed by Liu, Reiher, Sharifzadeh, and Staden, see Chapter 5 in~\cite{hong}, that translates finding the upper bound of $\varrho_5(12)$ into a weighted Tur\'an-type problem using Szemer\'edi's regularity lemma~\cite{MR1300968}. Then we use flag algebras to solve the weighted problem.
We give an outline of the method, the interested reader may read Chapter 5 in~\cite{hong} for a precise explanation.

First, fix a small $\varepsilon > 0$ and let
 $G$ be a (large) graph with $\alpha_5(G) < \varepsilon n$, i.e.,   every set of at least $\varepsilon n$ vertices in $G$ spans a copy of $K_5$. We want to show $G$ has edge density at most $\frac{10}{19}+o(1)$.
 
Next, we apply the regularity lemma to $G$ and obtain a cluster graph $R$.
Recall that each vertex $u$ in $R$ corresponds to a cluster of vertices $V_u$ in $G$.
To simplify the notation, we let $V_i$ to denote $V_{v_i}$.
For vertices $u$ and $v$ in $R$,  the \emph{density} of $uv$, denoted by $w(u,v)$, is the edge density between $V_u$ and $V_v$ in $G$.
Hence, we can view $R$ as an edge-weighted graph with $w$ being the weighting.

To turn the original problem into a weighted Tur\'an problem, we first show there are some forbidden weighted cliques (see Section~\ref{sec:otherw}). To find them, we use the following embedding lemma:

\begin{lemma} \label{lem:embedding}
    Let $\varepsilon > 0$ and $t \in \mathbb{N}$. Then, there exists an $\varepsilon' > 0$ such that the following holds. Let $V_1, \dots, V_t$ be clusters in an $\varepsilon'$-regular partition of a graph $G$ with $\alpha_p(G) = o(n)$. Suppose that, for each $i < t,$ we are given a permutation $\pi_i$ of $\{i+1, \dots, t\}$ and positive integers $s_{i,j},$ for $1 \le i \le j \le t$ satisfying $s_{i,i} = p$, for every  $i \in [t]$ and $s_{i,j} = \lceil s_{i,j-1} (d(V_i, V_{\pi_i(j)}) - \varepsilon/2) - \varepsilon/2 \rceil.$ Then, $G$ contains  a clique of size $\sum_{i=1}^t s_{i,t}$. 
\end{lemma}
\begin{proof}[Proof sketch.]
    Let $\varepsilon' \ll \varepsilon$ be chosen implicitly. Using standard cleaning arguments, we may assume that the pairs $(V_i, V_j)$ are $(\varepsilon'', d(V_i, V_j) - \varepsilon'')$-superregular for some $\varepsilon'' \ll \varepsilon$.
    Recall that a pair $(A,B)$ is \emph{$(\epsilon,\delta)$-superregular} if it is $\varepsilon$-regular, $\deg(a) \geq \delta|B|$ for all $a \in A$ and $\deg(b) \geq \delta|A|$ for all $b \in B$. 
    We will sequentially embed $s_{i,t}$ vertices into cluster $V_i$ such that the embedded vertices form a clique and their common neighbourhood has linear size in each of the subsequent clusters. Fix some $i \in [t]$ and suppose we have already embedded vertices into clusters $V_1, \ldots, V_{i-1}.$ By assumption, the embedded vertices have a linear sized common neighbourhood in $V_i$ and since $\alpha_p(G) = o(n),$ there is a $p$-clique $K$ in this set. Now, we find a subset of $K$ which has a linear sized common neighbourhood in the clusters $V_j$ for every $ j>i.$ To do this, we follow the ``recipe'' given to us by the permutation $\pi_i$ and sizes $s_{i,j}.$ Set $K_i = K.$  Suppose $i < j \le t$ and we have found $K_{j-1} \subseteq K$ such that the common neighbourhood of all vertices embedded into clusters $V_1, \dots, V_{i-1}$ and $K_{j-1}$ has linear size in each of the clusters $V_{\pi_i(i+1)}, \dots, V_{\pi_i(j-1)}.$ By superregularity, each vertex $v \in K_{j-1}$ has at least $d(V_i, V_{\pi_i(j)}) - \varepsilon''$ neighbours in $V_{\pi_i(j)}.$ Recalling the definition of $s_{i, j}$ it follows that there is a set $K_j \subseteq K_{j-1}$ whose common neighbourhood in $V_{\pi_i(j)}$ is of linear size. In the end, we have embedded $s_{i,t}$ vertices into cluster $V_i,$ yielding the desired clique.
\end{proof}

Plugging in specific values into Lemma~\ref{lem:embedding} can be somewhat opaque, so to make our proofs easier to follow, we write certain embedding rules. However, they are easily seen to be consequences of Lemma~\ref{lem:embedding}.

\begin{table}[h!]
\begin{center}
    \begin{tabular}{c|c|l}
    name/color & density interval   &   rule \\ \hline
     1    & $[0,\varepsilon)$     & no embedding \\ 
     2    & $[\varepsilon,1/5+\varepsilon)$     & any $1$ vertex \\ 
     3    & $[1/5+\varepsilon,1/2+\varepsilon)$    & some $2$ vertices \\ 
     4    & $[1/2+\varepsilon,3/5+\varepsilon)$  & any $2$ vertices or some $3$ vertices \\
     5    & $[3/5+\varepsilon,4/5+\varepsilon)$ & some $4$ vertices \\ 
     6   & $[4/5+\varepsilon,1]$  & any $5$ vertices\\ 
    \end{tabular}
\end{center}
\caption{Names, density intervals, and brief rules.}\label{tab:512}
\end{table}

In what follows, $v_i, v_j$ denote two vertices in the cluster graph and $w(v_i, v_j)$ denotes the density $d(V_i,V_j)$ in $G$.
\begin{itemize}
\item[(R0)] By assumptions on $\alpha_p(G)$,  $V_i$ contains $K_5$.

\item[(R1)] If $w(v_i,v_j) < \varepsilon$, then we cannot guarantee that vertices from both $V_i$ and $V_j$ can be used in an embedding of $K_{12}$.

\item[(R2)] If $w(v_i,v_j) \geq \varepsilon$, then arbitrary one vertex of $V_i$ can be used in an embedding of $K_{12}$.

\item[(R3)] If $w(v_i,v_j) \geq 1/5+\varepsilon$, then some two vertices of $V_i$ can be used in an embedding of $K_{12}$.

\item[(R4)] If $w(v_i,v_j) \geq 1/2+\varepsilon$, then 
 arbitrary two  vertices or some three vertices out of any $K_4 \subset K_5$ of $V_i$ can be used in an embedding of $K_{12}$.

\item[(R5)] If $w(v_i,v_j) \geq 3/5+\varepsilon$, then 
 some  $4$ vertices of the $K_5$  could be used in an embedding of $K_{12}$.

\item[(R6)]  If $w(v_i,v_j) \geq 4/5+\varepsilon$, then the $5$ vertices of the $K_5$ in $V_i$ can be used in an embedding of $K_{12}$.
\end{itemize}

A brief summary of the rules is in Table~\ref{tab:512}.
In justification for the rules, we want to show that there is a common neighborhood of linear size in  $V_j$.
If $w(v_i,v_j) = 0$, then there may be no edges at all between $V_i$ and $V_j$, so a clique on at least 2 vertices cannot intersect both $V_i$ and $V_j$. This gives (R1).
For (R2), if $w(v_i,v_j) > \varepsilon$, then any typical vertex in $V_i$ has linearly many neighbors in $V_j$.
For (R3) and (R5), first pick a $K_5$ in $V_i$ by (R0). By the pigeonhole principle, 2 and 4 vertices respectively have a common neighborhood of linear size in $V_j$. 
In (R4) and (R6), two and 5, respectively vertices must have a linear common neighborhood in $V_j$. 
For the second part of (R4), applying the pigeonhole principle gives that some three out of any 4 vertices have a common neighborhood of linear size in $V_j$.   


Notice that when considering $V_{i}$, we need to apply the rules towards all the remaining clusters. In particular, we can use a rule guaranteeing \emph{some} subset such as (R3) or (R5) only once and the remaining rules need to be valid for \emph{any} subset such as (R2) or (R6). For example, if we wish to embed 2 vertices in $V_j$, we can use (R3) to find 2 vertices $X$ that work for one of the remaining classes but for the rest, we need to use (R4) that guarantees $X$ has linear common neighborhood in all the remaining clusters.
  
\begin{figure}[h!]  
\begin{center}
    \begin{tikzpicture}[scale=2]
    \draw   
       (0,0)   node[xvtx,label=below:$v_1$](a){2}
       (120:1) node[xvtx,label=left:$v_2$](c){$5$} 
       (60:1)  node[xvtx,label=right:$v_3$](b){$5$} 
       (a) -- node[pos=0.5,right]{$\frac15+\varepsilon$} (b) 
       (a) -- node[pos=0.5,left]{$\frac12+\varepsilon$} (c) 
       (b) -- node[pos=0.5,above]{$\frac{4}{5}+\varepsilon$} (c)
       ;
    \draw (0,-0.5)node[below]{(a)};   
    \end{tikzpicture}
\hskip 1 em
    \begin{tikzpicture}[scale=2]
    \draw   
       (0,0)   node[xvtx,label=below:$v_1$](a){3}
       (60:1)  node[xvtx,label=right:$v_3$](b){$5$} 
       (120:1) node[xvtx,label=left:$v_2$](c){$4$} 
       (a) -- node[pos=0.5,right]{$\frac12+\varepsilon$} (b) 
       (a) -- node[pos=0.5,left]{$\frac{3}{5}+\varepsilon$} (c) 
       (b) -- node[pos=0.5,above]{$\frac{3}{5}+\varepsilon$} (c)
       ;
    \draw (0,-0.5)node[below]{(b)};   
    \end{tikzpicture}
\hskip 1em
    \begin{tikzpicture}[scale=2]
    \draw 
       (0,0) node[xvtx,label=below:$v_1$](1){1} 
       (1,0) node[xvtx,label=below:$v_2$](2){1} 
       (1,1) node[xvtx,label=above:$v_3$](3){$5$} 
       (0,1) node[xvtx,label=above:$v_4$](4){$5$} 
       (1) -- node[pos=0.5,below]{$\varepsilon$} (2) 
       (1) -- node[pos=0.3,left]{$\varepsilon$}  (3) 
       (1) -- node[pos=0.5,left]{$\varepsilon$}  (4)
       (2) -- node[pos=0.5,right]{$\varepsilon$} (3) 
       (2) -- node[pos=0.3,right]{$\varepsilon$}  (4) 
       (3) -- node[pos=0.5,above]{$\frac{4}{5}+\varepsilon$}  (4)
       ;
    \draw (0.5,-0.5)node[below]{(c)};   
    \end{tikzpicture} 
\\
\hskip 1em    
    \begin{tikzpicture}[scale=2]
    \draw 
       (0,0) node[xvtx,label=below:$v_1$](1){1} 
       (1,0) node[xvtx,label=below:$v_2$](2){2} 
       (1,1) node[xvtx,label=above:$v_3$](3){$4$} 
       (0,1) node[xvtx,label=above:$v_4$](4){$5$} 
       (1) -- node[pos=0.5,below]{$\varepsilon$} (2) 
       (1) -- node[pos=0.3,left]{$\varepsilon$}  (3) 
       (1) -- node[pos=0.5,left]{$\varepsilon$}  (4)
       (2) -- node[pos=0.7,right]{$\frac{1}{5}+\varepsilon$} (3) 
       (2) -- node[pos=0.3,right]{$\frac{1}{2}+\varepsilon$}  (4) 
       (3) -- node[pos=0.5,above]{$\frac{3}{5}+\varepsilon$}  (4)
       ;
    \draw (0.5,-0.5)node[below]{(d)};         
    \end{tikzpicture} 
   \hskip 1em  
    \begin{tikzpicture}[scale=2]
    \draw 
       (0,0) node[xvtx,label=below:$v_1$](1){2} 
       (1,0) node[xvtx,label=below:$v_2$](2){2} 
       (1,1) node[xvtx,label=above:$v_3$](3){$3$} 
       (0,1) node[xvtx,label=above:$v_4$](4){$5$} 
       (1) -- node[pos=0.5,below]{$\frac12+\varepsilon$} (2) 
       (1) -- node[pos=0.3,left]{$\frac{1}{5}+\varepsilon$}  (3) 
       (1) -- node[pos=0.5,left]{$\frac12+\varepsilon$}  (4)
       (2) -- node[pos=0.7,right]{$\frac{1}{5}+\varepsilon$} (3) 
       (2) -- node[pos=0.3,right]{$\frac{1}{2}+\varepsilon$}  (4) 
       (3) -- node[pos=0.5,above]{$\frac{1}{2}+\varepsilon$}  (4)
       ;
    \draw (0.5,-0.5)node[below]{(e)};         
    \end{tikzpicture}   
\end{center}
\caption{Forbidden weighted cliques $\mathcal{F}$ used in $\varrho_5(12) \leq \frac{10}{19}$.
The label on edge denotes the edge weight. The number in vertex $v_i$ indicates $k_i$, which is the number of vertices used from $V_i$ to obtain $K_{12}$. 
}\label{fig:K12}
\end{figure}
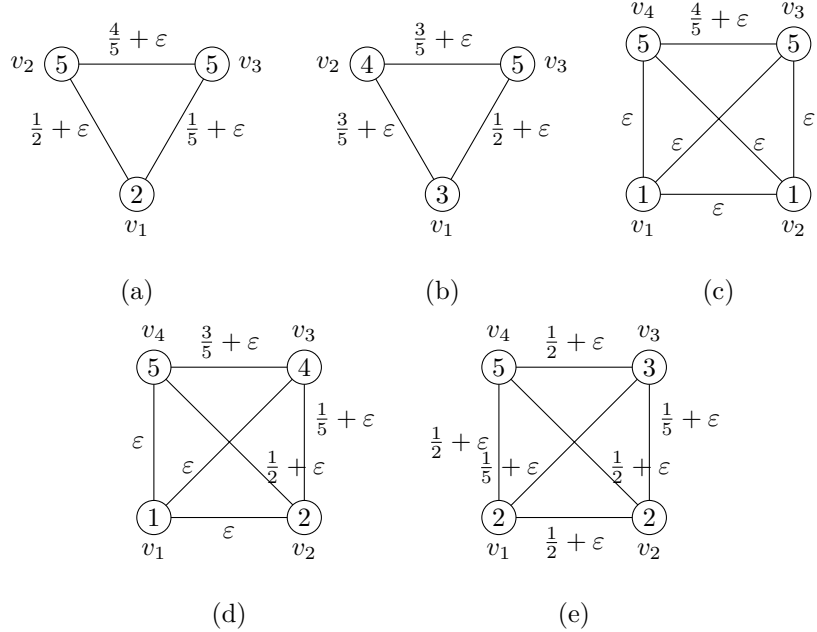

Let $\mathcal{F}$ be the weighted cliques depicted in Figure~\ref{fig:K12}.
Below we argue that if any of them appears in $R$, then we can find a $K_{12}$ in $G$. To simplify the explanation, we skip explicitly mentioning applications of (R0):\\
\begin{itemize}

\item[(a)] First we apply (R3) on $v_1v_3$ to find 2 vertices in $V_{1}$. The same two vertices work also for $v_1v_2$ by (R4). In $V_{2}$ we add a $K_5$ by (R6) on $v_2v_3$. Finally, we find another $K_5$ in $V_{3}$.

\item[(b)] Using (R5) on $v_1v_2$~we pick $4$ vertices from $V_1$ that have a linear sized common neighborhood in $V_2$. Using  (R4) on $v_1v_3$, among these $4$ vertices, we can choose $3$ that have a linear sized common neighborhood in $V_3$.
In $V_2$ we pick 4 vertices using (R5) on $v_2v_3$ and $V_3$ contributes the remaining $5$ vertices.
  
\item[(c)]  Using (R2) around $v_1$ and $v_2$, we choose one vertex from each $V_1$ and $V_2$. From each $V_3$ and $V_4$ we choose five vertices, which is possible by (R6) on $v_3v_4$.
  
\item[(d)] Using (R2) around $v_1$ we embed one vertex to $V_1$. Using (R3) on $v_2v_3$, we choose some two  vertices to $V_2$, and by (R4) the same two vertices work for  $v_2v_4$. In $V_3$ we add a $K_4$ by (R5) on $v_3v_4$, and we find $K_5$ in $V_4$.
  
\item[(e)] Using (R3) on $v_1v_3$ we choose some two vertices from $V_1$. The same two vertices work for $v_1v_2$ and $v_1v_4$ by (R4). By (R3) on $v_2v_3$, we embed some two adjacent vertices to $V_2$, and by (R4) the same two vertices work for  $v_2v_4$.
By (R4) on $v_3v_4$, we embed $3$ vertices in $V_3$. Finally, we embed a $K_5$ in $V_4$.
\end{itemize}

\medskip

Next we use flag algebras to find an upper bound of $d(\mathcal{F})$, see its definition in \eqref{eq:df}.
We simplify the problem to an edge-colored extremal problem. 
Let $H$ be a complete graph on the same vertex set as $R$.
Let $c:E(H)\to [6]$ be a $6$-edge-coloring defined according to Table~\ref{tab:512}.
Notice that forbidding $\mathcal{F}$ in $R$ gives a list of forbidden
$6$-edge-colored complete graphs $\mathcal{F}_c$.
Each configuration in Figure~\ref{fig:K12} gives several forbidden configurations in $\mathcal{F}_c$.

Let $c_i$ denote the density of edges of color $i$ in $H$. Given the above 6-edge-colored problem, flag algebras can be applied to prove an upper bound
\begin{align}
\frac{1}{5}c_2 + \frac{1}{2} c_3 + \frac{3}{5} c_4  + \frac{4}{5}c_5 + c_6 \leq \frac{10}{19}  + o(1).\label{eq:fa}
\end{align}
This implies $d(\mathcal{F}) \leq 10/19$.
Notice that when splitting densities into intervals, we used the lower bounds of the intervals to obtain forbidden configurations but we used the upper bounds to calculate the weight.
The calculation of \eqref{eq:fa} is computer assisted and it is too large to fit in the paper here. It can be downloaded at \oururl.
We use simple flag algebras approach, where we add a sum of squares to the left-hand side of \eqref{eq:fa} which yields the right-hand side. 
In addition to the bound 10/19, it would be possible to obtain a stability type result that the cluster graph is a blow-up of the graph depicted in Figure~\ref{fig:512}.
\end{proof}

\subsection{$\varrho_6(14) \leq \frac{12}{23}$}

Here $p=6$ and $q=14$, so in Conjecture~\ref{conj:false} we get $t=2, r=0$.
The three-partite construction depicted in Figure~\ref{fig:614-411}(a) gives $\frac{2p}{4p-1} = \frac{12}{23} \approx 0.5217391304$.

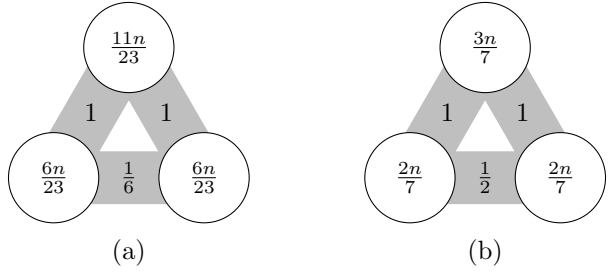
\begin{figure}
\begin{center}
       \begin{tikzpicture}[scale=2]
    \draw[line width=20 pt,color=gray!50!white] 
    (0,0) coordinate(a) -- (1,0) coordinate(b)
    (a)  -- (60:1)  coordinate(c)
    (c)--(b)
    ;
    \draw[fill=white]  
    (a) circle(0.3cm) (a) node{$\frac{6n}{23}$}
    (b) circle(0.3cm) (b) node{$\frac{6n}{23}$}
    (c) circle(0.3cm) (c) node{$\frac{11n}{23}$}
    ;
    \draw
    ($(a)!0.5!(b)$) node{$\frac{1}{6}$}
    ($(a)!0.5!(c)$) node{$1$}
    ($(c)!0.5!(b)$) node{$1$}
    ;
    \draw(0.5,-0.5) node{(a)};    
    \end{tikzpicture}
    \hskip 4em
           \begin{tikzpicture}[scale=2]
    \draw[line width=20 pt,color=gray!50!white] 
    (0,0) coordinate(a) -- (1,0) coordinate(b)
    (a)  -- (60:1)  coordinate(c)
    (c)--(b)
    ;
    \draw[fill=white]  
    (a) circle(0.3cm) (a) node{$\frac{2n}{7}$}
    (b) circle(0.3cm) (b) node{$\frac{2n}{7}$}
    (c) circle(0.3cm) (c) node{$\frac{3n}{7}$}
    ;
    \draw
    ($(a)!0.5!(b)$) node{$\frac{1}{2}$}
    ($(a)!0.5!(c)$) node{$1$}
    ($(c)!0.5!(b)$) node{$1$}
    ;
    \draw(0.5,-0.5) node{(b)};
    \end{tikzpicture}
\end{center}
\caption{Conjectured constructions  from Conjecture~\ref{conj:false} for (a) $\varrho_6(14) \geq \frac{12}{23}$ and (b) $\varrho_4(11) \geq \frac{4}{7}$.}\label{fig:614-411}
\end{figure}

The proof of the upper bound in this section is analogous to the proof in Section~\ref{sec:512}.
The main difference is that we use a different discretization of density intervals as in the table below, and different forbidden weighted cliques depicted in Figure~\ref{fig:K14}.

\begin{center}
    \begin{tabular}{c|c|l}
    name/color & weight interval   &  rule \\ \hline
     1    & $[0,\varepsilon)$     & no embedding \\ 
     2    & $[\varepsilon,1/6+\varepsilon)$   & any 1 vertex \\ 
     3    & $[1/6+\varepsilon,2/6+\varepsilon)$   & some 2 vertices \\ 
     4    & $[2/6+\varepsilon,3/6+\varepsilon)$   & some 3 vertices \\ 
     5    & $[3/6+\varepsilon,4/6+\varepsilon)$   & any 2 vertices or some 4 vertices \\ 
     6    & $[4/6+\varepsilon,5/6+\varepsilon)$   & some 5 vertices \\ 
     7    & $[5/6+\varepsilon,1]$   & any 6 vertices \\ 
    \end{tabular}
\end{center}

The proof is finished by using flag algebra calculation on $4$ vertices. The computation is computer assisted and it can be downloaded at \oururl.

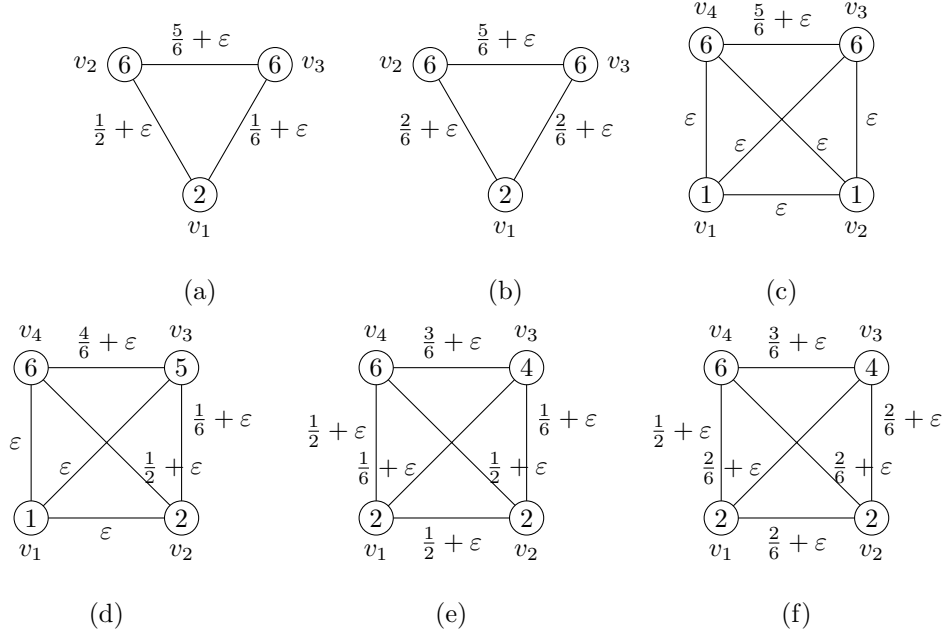
\begin{figure}
\begin{center}
    \begin{tikzpicture}[scale=2]
    \draw   
       (0,0)   node[xvtx,label=below:$v_1$](a){2}
       (120:1) node[xvtx,label=left:$v_2$](c){$6$} 
       (60:1)  node[xvtx,label=right:$v_3$](b){$6$} 
       (a) -- node[pos=0.5,right]{$\frac16+\varepsilon$} (b) 
       (a) -- node[pos=0.5,left]{$\frac12+\varepsilon$} (c) 
       (b) -- node[pos=0.5,above]{$\frac{5}{6}+\varepsilon$} (c)
       ;
    \draw (0,-0.5)node[below]{(a)};      
    \end{tikzpicture}
\hskip 1 em
    \begin{tikzpicture}[scale=2]
    \draw   
       (0,0)   node[xvtx,label=below:$v_1$](a){2}
       (120:1) node[xvtx,label=left:$v_2$](c){$6$} 
       (60:1)  node[xvtx,label=right:$v_3$](b){$6$} 
       (a) -- node[pos=0.5,right]{$\frac26+\varepsilon$} (b) 
       (a) -- node[pos=0.5,left]{$\frac26+\varepsilon$} (c) 
       (b) -- node[pos=0.5,above]{$\frac{5}{6}+\varepsilon$} (c)
       ;
    \draw (0,-0.5)node[below]{(b)};             
    \end{tikzpicture}
\hskip 1em
    \begin{tikzpicture}[scale=2]
    \draw 
       (0,0) node[xvtx,label=below:$v_1$](1){1} 
       (1,0) node[xvtx,label=below:$v_2$](2){1} 
       (1,1) node[xvtx,label=above:$v_3$](3){$6$} 
       (0,1) node[xvtx,label=above:$v_4$](4){$6$} 
       (1) -- node[pos=0.5,below]{$\varepsilon$} (2) 
       (1) -- node[pos=0.3,left]{$\varepsilon$}  (3) 
       (1) -- node[pos=0.5,left]{$\varepsilon$}  (4)
       (2) -- node[pos=0.5,right]{$\varepsilon$} (3) 
       (2) -- node[pos=0.3,right]{$\varepsilon$}  (4) 
       (3) -- node[pos=0.5,above]{$\frac{5}{6}+\varepsilon$}  (4)
       ;
    \draw (0.5,-0.5)node[below]{(c)};          
    \end{tikzpicture} 
  
    \begin{tikzpicture}[scale=2]
    \draw 
       (0,0) node[xvtx,label=below:$v_1$](1){1} 
       (1,0) node[xvtx,label=below:$v_2$](2){2} 
       (1,1) node[xvtx,label=above:$v_3$](3){$5$} 
       (0,1) node[xvtx,label=above:$v_4$](4){$6$} 
       (1) -- node[pos=0.5,below]{$\varepsilon$} (2) 
       (1) -- node[pos=0.3,left]{$\varepsilon$}  (3) 
       (1) -- node[pos=0.5,left]{$\varepsilon$}  (4)
       (2) -- node[pos=0.7,right]{$\frac{1}{6}+\varepsilon$} (3) 
       (2) -- node[pos=0.3,right]{$\frac{1}{2}+\varepsilon$}  (4) 
       (3) -- node[pos=0.5,above]{$\frac{4}{6}+\varepsilon$}  (4)
       ;
    \draw (0.5,-0.5)node[below]{(d)};                 
    \end{tikzpicture} 
   \hskip 1em  
    \begin{tikzpicture}[scale=2]
    \draw 
       (0,0) node[xvtx,label=below:$v_1$](1){2} 
       (1,0) node[xvtx,label=below:$v_2$](2){2} 
       (1,1) node[xvtx,label=above:$v_3$](3){$4$} 
       (0,1) node[xvtx,label=above:$v_4$](4){$6$} 
       (1) -- node[pos=0.5,below]{$\frac12+\varepsilon$} (2) 
       (1) -- node[pos=0.3,left]{$\frac{1}{6}+\varepsilon$}  (3) 
       (1) -- node[pos=0.6,left]{$\frac12+\varepsilon$}  (4)
       (2) -- node[pos=0.7,right]{$\frac{1}{6}+\varepsilon$} (3) 
       (2) -- node[pos=0.3,right]{$\frac{1}{2}+\varepsilon$}  (4) 
       (3) -- node[pos=0.5,above]{$\frac{3}{6}+\varepsilon$}  (4)
       ;
    \draw (0.5,-0.5)node[below]{(e)};                        
    \end{tikzpicture}  
       \hskip 1em  
    \begin{tikzpicture}[scale=2]
    \draw 
       (0,0) node[xvtx,label=below:$v_1$](1){2} 
       (1,0) node[xvtx,label=below:$v_2$](2){2} 
       (1,1) node[xvtx,label=above:$v_3$](3){$4$} 
       (0,1) node[xvtx,label=above:$v_4$](4){$6$} 
       (1) -- node[pos=0.5,below]{$\frac26+\varepsilon$} (2) 
       (1) -- node[pos=0.3,left]{$\frac{2}{6}+\varepsilon$}  (3) 
       (1) -- node[pos=0.6,left]{$\frac12+\varepsilon$}  (4)
       (2) -- node[pos=0.7,right]{$\frac{2}{6}+\varepsilon$} (3) 
       (2) -- node[pos=0.3,right]{$\frac{2}{6}+\varepsilon$}  (4) 
       (3) -- node[pos=0.5,above]{$\frac{3}{6}+\varepsilon$}  (4)
       ;
    \draw (0.5,-0.5)node[below]{(f)};                               
    \end{tikzpicture}  
\end{center}
\caption{Forbidden configurations used in $\varrho_6(14) \leq \frac{4}{7}$.}\label{fig:K14}
\end{figure}

\subsection{$\varrho_4(11) \leq \frac{4}{7}$}

Here $p=4$ and $q=11$, so in Conjecture~\ref{conj:false} we get $t=2, r=1$.
The three-partite construction depicted in Figure~\ref{fig:614-411}(b) gives $\frac47$.

The proof of the upper bound in this section is analogous to the proof in Section~\ref{sec:512}.
The main difference is that we use a different discretization of density intervals as in the table below, and different forbidden weighted cliques depicted in Figure~\ref{fig:K11}.

\begin{center}
    \begin{tabular}{c|c|l}
    name/color & density interval   &   rule \\ \hline
     1    & $[0,\varepsilon)$     & no embedding \\ 
     2    & $[\varepsilon,1/4+\varepsilon)$     & any $1$ vertex \\ 
     3    & $[1/4+\varepsilon,1/2+\varepsilon)$    & some $2$ vertices \\ 
     4    & $[1/2+\varepsilon,3/4+\varepsilon)$  & any $2$ vertices or some 3 vertices\\
     5    & $[3/4+\varepsilon,1]$ & any $4$ vertices\\ 
    \end{tabular}
\end{center}

The proof is finished by using flag algebra calculation on $5$ vertices. The computation is computer assisted and it can be downloaded at \oururl.

\begin{figure}[h!]
\begin{center}
    \begin{tikzpicture}[scale=2]
    \draw 
       (0,0) node[xvtx,label=below:$v_1$](a){4} (60:1) node[xvtx,label=right:$v_3$](b){4} 
       (a) (120:1) node[xvtx,label=left:$v_2$](c){3} 
       (a) -- node[pos=0.5,right]{$\frac34+\varepsilon$} (b) 
       (a) -- node[pos=0.5,left]{$\frac34+\varepsilon$} (c) 
       (b) -- node[pos=0.5,above]{$\frac12+\varepsilon$} (c)
       ;
           \draw (0,-0.5)node[below]{(a)};          
    \end{tikzpicture}
\hskip 0.5em
    \begin{tikzpicture}[scale=2]
    \draw 
       (0,0) node[xvtx,label=below:$v_1$](1){1} 
       (1,0) node[xvtx,label=below:$v_2$](2){2} 
       (1,1) node[xvtx,label=above:$v_3$](3){4} 
       (0,1) node[xvtx,label=above:$v_4$](4){4} 
       (1) -- node[pos=0.5,below]{$\varepsilon$} (2) 
       (1) -- node[pos=0.3,left]{$\varepsilon$}  (3) 
       (1) -- node[pos=0.5,left]{$\varepsilon$}  (4)
       (2) -- node[pos=0.5,right]{$\frac14+\varepsilon$} (3) 
       (2) -- node[pos=0.3,right]{$\frac12+\varepsilon$}  (4) 
       (3) -- node[pos=0.5,above]{$\frac34+\varepsilon$}  (4)
       ;
           \draw (0.5,-0.5)node[below]{(b)};         
    \end{tikzpicture}   
\hskip 0.5em
    \begin{tikzpicture}[scale=2]
    \draw 
       (0,0) node[xvtx,label=below:$v_1$](1){2} 
       (1,0) node[xvtx,label=below:$v_2$](2){2} 
       (1,1) node[xvtx,label=above:$v_3$](3){3} 
       (0,1) node[xvtx,label=above:$v_4$](4){4} 
       (1) -- node[pos=0.5,below]{$\frac14+\varepsilon$} (2) 
       (1) -- node[pos=0.3,left]{$\frac12+\varepsilon$}  (3) 
       (1) -- node[pos=0.5,left]{$\frac12+\varepsilon$}  (4)
       (2) -- node[pos=0.5,right]{$\frac14+\varepsilon$} (3) 
       (2) -- node[pos=0.3,right]{$\frac12+\varepsilon$}  (4) 
       (3) -- node[pos=0.5,above]{$\frac34+\varepsilon$}  (4)
       ;
           \draw (0.5,-0.5)node[below]{(c)};         
    \end{tikzpicture}   
    \hskip 0.5em
    \begin{tikzpicture}[scale=1.8]
    \draw
       (90:1) node[xvtx,label=above:$v_1$](1){1} 
       (90+72:1) node[xvtx,label=left:$v_2$](2){1} 
       (90+2*72:1) node[xvtx,label=below:$v_3$](3){1} 
       (90+3*72:1) node[xvtx,label=below:$v_4$](4){4} 
       (90+4*72:1) node[xvtx,label=right:$v_5$](5){4} 
       (1) -- node[pos=0.5,below]{$\varepsilon$} (2) 
       (1) -- node[pos=0.2,left]{$\varepsilon$}  (3) 
       (1) -- node[pos=0.2,left]{$\varepsilon$}  (4)
       (1) -- node[pos=0.5,left]{$\varepsilon$}  (5)
       (2) -- node[pos=0.5,left]{$\varepsilon$} (3) 
       (2) -- node[pos=0.2,below]{$\varepsilon$}  (4) 
       (2) -- node[pos=0.2,below]{$\varepsilon$}  (5) 
       (3) -- node[pos=0.5,below]{$\varepsilon$}  (4)
       (3) -- node[pos=0.5,above]{$\varepsilon$}  (5)
       (4) -- node[pos=0.5,right]{$\frac34+\varepsilon$}  (5)
       ;
           \draw (0,-1.2)node[below]{(d)};                
    \end{tikzpicture}   
    \hskip 0.5em
    \begin{tikzpicture}[scale=1.8]
    \draw
       (90:1) node[xvtx,label=above:$v_1$](1){1} 
       (90+72:1) node[xvtx,label=left:$v_2$](2){1} 
       (90+2*72:1) node[xvtx,label=below:$v_3$](3){2} 
       (90+3*72:1) node[xvtx,label=below:$v_4$](4){3} 
       (90+4*72:1) node[xvtx,label=right:$v_5$](5){4} 
       (1) -- node[pos=0.5,below]{$\varepsilon$} (2) 
       (1) -- node[pos=0.2,left]{$\varepsilon$}  (3) 
       (1) -- node[pos=0.2,left]{$\varepsilon$}  (4)
       (1) -- node[pos=0.5,left]{$\varepsilon$}  (5)
       (2) -- node[pos=0.5,left]{$\varepsilon$} (3) 
       (2) -- node[pos=0.2,below]{$\varepsilon$}  (4) 
       (2) -- node[pos=0.2,below]{$\varepsilon$}  (5) 
       (3) -- node[pos=0.5,below]{$\frac14+\varepsilon$}  (4)
       (3) -- node[pos=0.5,above]{$\frac12+\varepsilon$}  (5)
       (4) -- node[pos=0.5,right]{$\frac12+\varepsilon$}  (5)
       ;
           \draw (0,-1.2)node[below]{(e)};                
    \end{tikzpicture}
\end{center}
\caption{Forbidden configurations used in $\varrho_4(11) \leq \frac{4}{7}$.}\label{fig:K11}
\end{figure}
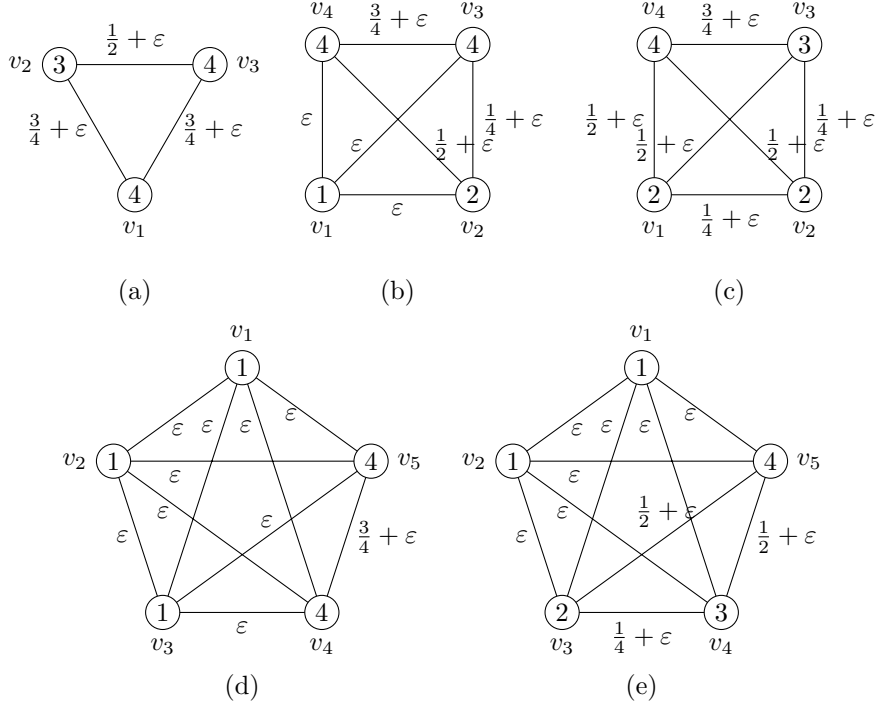



\subsection{Upper bounds for $\varrho_{p}(p+6)$}

In this section we show
\[
\varrho_{p}(p+6) \leq  \begin{cases}
   \frac{5}{2p}  & \text{ if } p \leq 12 \\
    \frac{8}{3p} & \text{ otherwise. }
    \end{cases}
\]
We mentioned earlier that  $\varrho_3(9)\le 3/5$, which is better than our bound.
Note that for $p=4$, \cite{hong} proved $\varrho_4(10) = \frac{8}{15}$, which is also better than our bound. For $p = 5$, Erd\H{o}s~\cite{MR1300968} has $\varrho_5(11)= \frac{5}{10}$, which is the same as our bound.
The case $p \geq 13$ is implied by Theorem 1.5 from~\cite{hong} by setting $s=4$  and $t=3$. We reprove it here using our method.
In the following we assume $p \geq 6$.

The two conjectured extremal constructions are 
a balanced bipartite graph with edge density  $\frac{5}{p}$ between the two parts giving
edge density $\frac{5}{2p}$,
and a balanced tripartite graph with edge density $\frac{4}{p}$ between every two parts, giving
edge density $\frac{8}{3p}$.
Notice that these are not necessarily realizable since one has to fill in the parts graphs, where every independent set of $\alpha n$ vertices contain $K_{p}$ but there is no $K_{p+6}$ in the resulting graph. It is not obvious how to make such a construction.

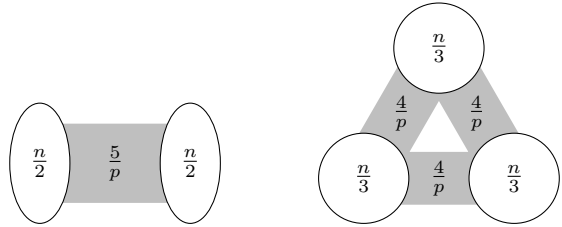
\begin{center}
    \begin{tikzpicture}[scale=2]
    \draw[line width=30 pt,color=gray!50!white] 
    (0,0) -- (1,0)
    ;
    \draw[fill=white]
    (0,0) ellipse(0.2 cm and 0.4 cm)
    (1,0) ellipse(0.2 cm and 0.4 cm)
    ;
    \draw
    (0,0) node{$\frac{n}{2}$}
    (1,0) node{$\frac{n}{2}$}
    (0.5,0) node{$\frac{5}{p}$}
    ;
    \end{tikzpicture}
\hskip 3em    
       \begin{tikzpicture}[scale=2]
    \draw[line width=20 pt,color=gray!50!white] 
    (0,0) coordinate(a) -- (1,0) coordinate(b)
    (a)  -- (60:1)  coordinate(c)
    (c)--(b)
    ;
    \foreach \x in {a,b,c}
    {
    \draw[fill=white]
    (\x) circle(0.3cm)
    (\x) node{$\frac{n}{3}$}
    ;
    }
    \draw
    ($(a)!0.5!(b)$) node{$\frac{4}{p}$}
    ($(a)!0.5!(c)$) node{$\frac{4}{p}$}
    ($(c)!0.5!(b)$) node{$\frac{4}{p}$}
    ;
    \end{tikzpicture}
\end{center}

Notice that the tripartite construction has a higher edge density. 
Below, We rule out the tripartite construction for $p \leq 12$ while Liu, Reiher, Sharifzadeh, and Staden~\cite{hong} actually constructed it for $p=16$, and commented that their method should work for some other values of $p$. 

Now we focus on the upper bound.
First observe that there is no edge in the cluster graph with weight at least $5/p+\varepsilon$.
If there was one, we could find $6 + p$ vertices forming a clique, which is a contradiction.

The rest of proof of the upper bound is again analogous to the proof in Section~\ref{sec:512}.
We discretize the densities according to the following table. Notice that the table works for all $p$ since we have an upper bound of $5/p+\varepsilon$ on the densities.
\begin{center}
    \begin{tabular}{c|c|l}
    name/color & weight interval   &  rule \\ \hline
     1    & $[0,\varepsilon)$   & no embedding  \\ 
     2    & $[\varepsilon,1/p+\varepsilon)$   & any 1 vertex \\ 
     3    & $[1/p+\varepsilon,2/p+\varepsilon)$   & some 2 vertices \\ 
     4    & $[2/p+\varepsilon,3/p+\varepsilon)$   & some 3 vertices  \\
     5    & $[3/p+\varepsilon,4/p+\varepsilon)$   & some 4 vertices \\ 
     6    & $[4/p+\varepsilon, 5/p+\varepsilon)$   & some 5 vertices \\ 
    \end{tabular}
\end{center}

The forbidden configurations used for showing $\varrho_p(p+6) \leq \frac{8}{3p}$ are depicted in Figure~\ref{fig:83}.

\begin{figure}[h!]
\begin{center}
    \begin{tikzpicture}[scale=2]
    \draw 
       (0,0) node[xvtx,label=below:$v_1$](a){1} (60:1) node[xvtx,label=right:$v_3$](b){$p$} 
       (a) (120:1) node[xvtx,label=left:$v_2$](c){5} 
       (a) -- node[pos=0.5,right]{$\varepsilon$} (b) 
       (a) -- node[pos=0.5,left]{$\varepsilon$} (c) 
       (b) -- node[pos=0.5,above]{$\frac4p+\varepsilon$} (c)
       ;
           \draw (0,-0.5)node[below]{(a)};          
    \end{tikzpicture}
\hskip 0.5em
    \begin{tikzpicture}[scale=2]
    \draw 
       (0,0) node[xvtx,label=below:$v_1$](1){1} 
       (1,0) node[xvtx,label=below:$v_2$](2){1} 
       (1,1) node[xvtx,label=above:$v_3$](3){4} 
       (0,1) node[xvtx,label=above:$v_4$](4){$p$} 
       (1) -- node[pos=0.5,below]{$\varepsilon$} (2) 
       (1) -- node[pos=0.3,left]{$\varepsilon$}  (3) 
       (1) -- node[pos=0.5,left]{$\varepsilon$}  (4)
       (2) -- node[pos=0.5,right]{$\varepsilon$} (3) 
       (2) -- node[pos=0.3,right]{$\varepsilon$}  (4) 
       (3) -- node[pos=0.5,above]{$\frac3p+\varepsilon$}  (4)
       ;
           \draw (0.5,-0.5)node[below]{(b)};         
    \end{tikzpicture}   
\hskip 0.5em
    \begin{tikzpicture}[scale=1.8]
    \draw
       (90:1) node[xvtx,label=above:$v_1$](1){1} 
       (90+72:1) node[xvtx,label=left:$v_2$](2){1} 
       (90+2*72:1) node[xvtx,label=below:$v_3$](3){1} 
       (90+3*72:1) node[xvtx,label=below:$v_4$](4){3} 
       (90+4*72:1) node[xvtx,label=right:$v_5$](5){$p$} 
       (1) -- node[pos=0.5,below]{$\varepsilon$} (2) 
       (1) -- node[pos=0.2,left]{$\varepsilon$}  (3) 
       (1) -- node[pos=0.2,left]{$\varepsilon$}  (4)
       (1) -- node[pos=0.5,left]{$\varepsilon$}  (5)
       (2) -- node[pos=0.5,left]{$\varepsilon$} (3) 
       (2) -- node[pos=0.2,below]{$\varepsilon$}  (4) 
       (2) -- node[pos=0.2,below]{$\varepsilon$}  (5) 
       (3) -- node[pos=0.5,below]{$\varepsilon$}  (4)
       (3) -- node[pos=0.5,above]{$\varepsilon$}  (5)
       (4) -- node[pos=0.5,right]{$\frac2p+\varepsilon$}  (5)
       ;
           \draw (0,-1.2)node[below]{(c)};                
    \end{tikzpicture}   
\end{center}
\caption{Forbidden configurations used in $\varrho_p(p+6) \leq \frac{8}{3p}$.}\label{fig:83}
\end{figure}
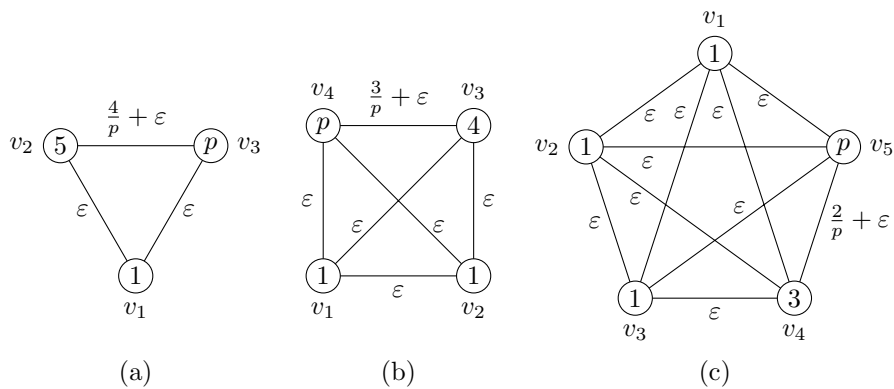

\begin{figure}
\begin{center}
    \begin{tikzpicture}[scale=2]
    \draw 
       (0,0) node[xvtx,label=below:$v_1$](a){2} (60:1) node[xvtx,label=right:$v_3$](b){$p$} 
       (a) (120:1) node[xvtx,label=left:$v_2$](c){4} 
       (a) -- node[pos=0.5,right]{$\frac3p+\varepsilon$} (b) 
       (a) -- node[pos=0.5,left]{$\frac3p+\varepsilon$} (c) 
       (b) -- node[pos=0.5,above]{$\frac3p+\varepsilon$} (c)
       ;
    \end{tikzpicture}
\end{center}
\caption{An extra forbidden configuration in addition to Figure~\ref{fig:83} used to show $\varrho_p(p+6) \leq \frac{5}{2p}$.}\label{fig:52}
\end{figure}
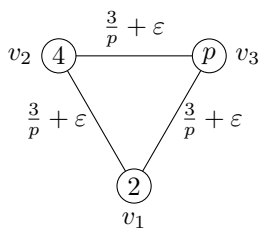

For $p \leq 12$, we can add one more forbidden configuration to the list that results in an improved upper bound, see Figure~\ref{fig:52}.
The only trick is with $v_1$. First, we pick $4$ vertices that have linear common neighborhood in $V_3$. From these 4 vertices, we can pick 2 by pigeonhole principle that have linear common neighborhood in $V_2$. These are the two vertices we keep.
The second step works only if $ 4 \cdot (\frac{3}{p} + \varepsilon) > 1$, which gives the restriction $p \leq 12$. This provides bound $\varrho_p(p+6) \leq \frac{5}{2p}$.
The computation is computer assisted and it can be downloaded at \oururl.

\section*{Acknowledgment}
The authors are grateful for Ce Chen for reading carefully the manuscript. The second author thanks Benny Sudakov for suggesting working on proving Theorem~\ref{thm:standard_weighting}, who heard the problem from the first author at a conference in Oberwolfach.
This work used the computing resources at the Center for Computational Mathematics, University of Colorado Denver, including the Alderaan cluster, supported by the National Science Foundation award OAC-2019089.

\bibliographystyle{plainurl}
\bibliography{references}

\end{document}